\documentclass[reqno, 10pt]{amsart}
\usepackage{amssymb}
\usepackage{amsmath}
\usepackage[mathscr]{euscript}
\usepackage{hyperref}
\usepackage{graphicx}
\usepackage[normalem]{ulem}

\usepackage{amssymb}
\usepackage{hyperref}
\RequirePackage[dvipsnames]{xcolor} 
\definecolor{halfgray}{gray}{0.55} 
\definecolor{webgreen}{rgb}{0,0.5,0}
\definecolor{webbrown}{rgb}{.6,0,0} \hypersetup{%
  colorlinks=true, linktocpage=true, pdfstartpage=3,
  pdfstartview=FitV,%
  breaklinks=true, pdfpagemode=UseNone, pageanchor=true,
  pdfpagemode=UseOutlines,%
  plainpages=false, bookmarksnumbered, bookmarksopen=true,
  bookmarksopenlevel=1,%
  hypertexnames=true,
  pdfhighlight=/O,
  urlcolor=webbrown, linkcolor=RoyalBlue,
  citecolor=webgreen, 
  pdftitle={},%
  pdfauthor={},%
  pdfsubject={2000 MAthematical Subject Classification: Primary:},%
  pdfkeywords={},%
  pdfcreator={pdfLaTeX},%
  pdfproducer={LaTeX with hyperref}%
}

\usepackage{enumitem}

\theoremstyle{plain}
\newtheorem{theorem}{Theorem}[section]

\newtheorem{corollary}[theorem]{Corollary}

\theoremstyle{definition}

\newtheorem{remark}[theorem]{Remark}
\newtheorem{example}[theorem]{Example}

\def\Z{\mathbb{Z}}
\def\R{\mathbb{R}}

\begin{document}

\title{Linearization and H\" older Continuity for Nonautonomous Systems}

\begin{abstract} We consider a nonautonomous system 
	\[ \dot x=A(t)x+f(t,x,y),\quad \dot y = g(t,y)\]
	and give conditions under which there is a transformation of the form
	$H(t,x,y)$
	$=(x+h(t,x,y),y)$ taking its solutions onto the solutions of
	the partially linearized system
	\[ \dot x=A(t)x,\quad \dot y = g(t,y).\]
	 Shi and Xiong \cite{SX} proved a special case where $g(t,y)$ was a linear function of 
	 $y$ and $\dot x=A(t)x$ had an exponential dichotomy. Our assumptions on $A$ and $f$ are of the general form considered by
	 Reinfelds and Steinberga \cite{RS}, which include many of the generalizations
	 of Palmer's theorem proved by other authors. Inspired by the work of Shi and Xiong,
	 we also prove H\" older continuity of $H$ and its inverse in $x$ and $y$.
	 Again the proofs are given in the context of Reinfelds and Steinberga but we 
	 show what the results reduce to when $\dot x=A(t)x$ is assumed to have an exponential
	 dichotomy. The paper is concluded with the discrete version of the results. 

\end{abstract}

\author{Lucas Backes}
\address{\noindent Departamento de Matem\'atica, Universidade Federal do Rio Grande do Sul, Av. Bento Gon\c{c}alves 9500, CEP 91509-900, Porto Alegre, RS, Brazil.}
\email{lucas.backes@ufrgs.br} 

\author{Davor Dragi\v cevi\'c}
\address{Department of Mathematics, University of Rijeka, Croatia}
\email{ddragicevic@math.uniri.hr}

\author{Kenneth J. Palmer}
\address{Department of Mathematics, National Taiwan University, Taipei, Taiwan}
\email{palmer@math.ntu.edu.tw}

\date{\today}

\maketitle

\section{Introduction}
One of the basic questions in the qualitative theory of dynamical systems is whether a nonlinear system (in  a neighborhood of its equilibrium) is  equivalent to its linear part. The celebrated Grobman-Hartman theorem (see~\cite{G1,G2,H1,H2}) asserts 
that if $x_0$ is a hyperbolic fixed point of a $C^1$-diffeomorphism $F\colon \mathbb R^d \to \mathbb R^d$ (i.e. the spectrum of $DF(x_0)$ does not intersect the unit circle in $\mathbb C$), then there exists a neighborhood $U$ of $x_0$ such that $F$ on $U$ is topologically conjugated to
$DF(x_0)$. In addition, this result also has  its global version. Namely, if $A\colon \mathbb R^d \to \mathbb R^d$ is a hyperbolic automorphism and $f\colon \R^d \to \R^d$ is a bounded Lipschitz map whose Lipschitz constant is sufficiently small, then $A$ and $A+f$ are topologically conjugated on $\mathbb R^d$.
Furthermore, there are analogous results for the case of flows. These results were extended to the case of  Banach spaces  independently by Palis~\cite{Palis} and Pugh~\cite{Pugh}, who also greatly simplified the original arguments of Grobman and Hartman.

It is well-known that even if $F$ is $C^\infty$, the conjugacy can fail to be locally Lipschitz. In fact, the conjugacy is in general only locally  H\" older continuous. Although this fact was apparently known to experts for some time, to the best of our knowledge the first written proof appeared in~\cite{SX} (in a more general setting that we  describe below). More recently, the same results was essentially reproved in~\cite{BV}. We stress that many works were devoted to the problem of formulating sufficient conditions which would ensure that the conjugacy  exhibits higher regularity properties. In this direction, we refer to the seminal works of Sternberg~\cite{Stern} and Belitskii~\cite{Bel73,Bel78},
as well as to some more recent contributions~\cite{E1,E2, R-S-JDDE04, R-S-JDDE06, ZhangZhang14JDE,ZZJ} and references therein.

The first version of the Grobman-Hartman theorem for nonautonomous dynamics is due to Palmer~\cite{Palmer}. In order to describe this result, let us consider a nonlinear and nonautonomous differential equation
\begin{equation}\label{j1}
x'=A(t)x+f(t,x),
\end{equation}
where $A\colon \mathbb R \to M_d$ and $f\colon \mathbb R\times \mathbb R^d\to \mathbb R^d$ are continuous maps. Here, $M_d$ denotes the space of all real matrices of order $d$. Furthermore, let 
\begin{equation}\label{j2}
x'=A(t)x,
\end{equation}
be the corresponding linear part. Assume  that~\eqref{j2} admits an exponential dichotomy (see Subsection~\ref{EX}) and that $f$ is bounded and Lipschitz in $x$ with a sufficiently small Lipschitz constant. Under these assumptions, Palmer proved that~\eqref{j1} and~\eqref{j2} are 
topologically conjugated. We note that an analogous result for the case of discrete time was obtained by Aulbach and Wanner~\cite{AW}. In a similar manner to that in the case of autonomous dynamics, the conjugacies in Palmer's theorem are in  general only locally H\" older continuous~\cite{SX,BV}.
Recently, several authors formulated sufficient conditions under which the conjugacies exhibit higher regularity. We refer to~\cite{CR, CDS, DZZ, DZZ2} and references therein.

In addition, several authors obtained important extensions of the Palmer's theorem by relaxing some of its assumptions.  In particular, Lin~\cite{Lin} discussed the case when $f$ can be unbounded and~\eqref{j2} is asymptotically stable. Moreover, Jiang considered the case when~\eqref{j2} admits either ordinary or a certain
general type of dichotomy~\cite{Jiang, Jiang2}. Finally, Reinfelds and Steinberga~\cite{RS} presented a rather general linearization result than can be applied to situation when~\eqref{j2} does not possess any type of dichotomy and that includes the main results from~\cite{Jiang, Jiang2, Palmer} as a particular case.
We also note that the approach developed in~\cite{RS} was used in~\cite{BD} to establish sufficient conditions under which~\eqref{j1} and~\eqref{j2} are topologically equivalent,  when~\eqref{j2} admits  the so-called generalized exponential dichotomy. However, in~\cite{BD} the authors in addition show that the conjugacies are unique in a 
suitable class.

An interesting extension of Palmer's theorem was proposed by Shi and Xiong~\cite{SX}. Let us consider the following coupled system
\begin{equation}\label{j3}
x'=A(t)x+f(t,x, y), \quad y'=B(t)y,
\end{equation}
where $A\colon \mathbb R\to M_{d_1}$, $B\colon \mathbb R\to M_{d_2}$ and $f\colon \mathbb R\times \mathbb R^{d_1}\times \mathbb R^{d_2}\to \mathbb R^{d_1}$ are continuous maps. Furthermore, we consider the associated linear system
\begin{equation}\label{j4}
x'=A(t)x, \quad y'=B(t)y.
\end{equation}
Assume that $x'=A(t)x$ admits an exponential dichotomy and that $f$ is bounded and Lipschitz in $(x, y)$ with a sufficiently small Lipschitz constant. Under these assumptions, it is proved in~\cite{SX} that~\eqref{j3} and~\eqref{j4} are topologically conjugated and that the corresponding conjugacies are locally  H\" older continuous.

The goal of the present paper is to essentially combine ideas from~\cite{RS} and~\cite{SX}. More precisely, following~\cite{RS} we formulate new  sufficient conditions under which~\eqref{j3} and~\eqref{j4} are topologically conjugated (by also allowing the second equation in~\eqref{j3} and~\eqref{j4} to be nonlinear).
Our basic linearization result (see Theorem~\ref{t1}) does not require that the first component in~\eqref{j4} admit an exponential dichotomy (or even ordinary dichotomy). We also formulate sufficient conditions for the  H\" older continuity of the conjugacies in both variables $x$ and $y$. We stress that 
 H\" older linearization was not  previously discussed in the setting of~\cite{RS}. Finally, we note that in the particular case when the first equation in~\eqref{j4} admits an exponential dichotomy, our  H\" older linearization result improves that from~\cite{SX}  (see Remark~\ref{Obs}).

The paper is organized as follows. In Section~\ref{CT} we consider the case of continuous time
and first formulate our general linearization result whose proof follows closely the ideas from~\cite{RS}. Afterwards, we discuss the H\" older continuity of the conjugacies in both variables separately. We conclude by comparing 
our H\" older linearization results with that from~\cite{SX}. Finally, in Section~\ref{DT} we obtain the corresponding results for the case of discrete time.

\section{The case of continuous time}\label{CT}
\subsection{Preliminaries}\label{Pre}
Let $(X, |\cdot |_X)$ and $(Y, |\cdot |_Y)$ be two arbitrary Banach spaces. For the sake of simplicity both norms $|\cdot |_X$ and $|\cdot |_Y$ will be denoted simply by $|\cdot |$. By $\mathcal B(X)$ we denote the space of all bounded operators on $X$ equipped with the operator norm (which we will also denote by 
$|\cdot |$).

Let $A\colon \R \to \mathcal B(X)$ be such that $t\mapsto A(t)$ is continuous. Furthermore, let $f\colon \R \times X\times Y\to X$ be a continuous map with the property that there exist continuous functions $\mu, \gamma \colon \R \to [0,\infty)$ such that 
\begin{equation}\label{feq}
|f(t, x, y)|\le \mu(t) \quad \text{and} \quad |f(t,x_1, y)-f(t, x_2, y)|\le \gamma(t) |x_1-x_2|, 
\end{equation}
for $t\in \R$, $x, x_1, x_2\in X$ and $y\in Y$. Finally, let $g\colon \R \times Y \to Y$ be a continuous map and suppose  that solutions of 
$\dot y=g(t,y)$ are defined for all time and that there exists a unique such solution
$t\mapsto y(t)$ such that $y(\tau)=\eta$ for any given pair $(\tau,\eta)\in \R \times Y$. We consider a coupled system
\begin{equation}\label{Heq2} \dot x=A(t)x+f(t,x,y),\quad \dot y=g(t,y).\end{equation}
In Appendix~\ref{A1} we show that our conditions imply
that the solutions of~\eqref{Heq2} are defined for all time and that there exists a unique such solution
$t \mapsto (x(t),y(t))$ such that $(x(\tau),y(\tau))=(\xi,\eta)$ for any given triple $(\tau,\xi,\eta)\in \R \times X \times Y$.

By $T(t,s)$ we will denote the evolution family associated to the equation $\dot x=A(t)x$. Furthermore, let $P\colon \R \to \mathcal B(X)$ be an arbitrary Bochner measurable map. We define
\begin{equation}\label{G}
\mathcal G(t,s)=\begin{cases}
T(t,s)P(s) & \text{for $t\ge s$,}\\
-T(t,s)(I-P(s)) & \text{for $t<s$,}
\end{cases}
\end{equation}
where $I$ denotes the identity operator on $X$.
\subsection{A linearization result}
Besides~\eqref{Heq2}, we also consider the uncoupled system
\begin{equation}\label{Heq1} \dot x=A(t)x,\quad \dot y=g(t,y).\end{equation}
We denote by $t\mapsto (x_1(t,\tau,\xi), y(t,\tau,\eta))$ the solution of (\ref{Heq1})
such that $x(\tau)=\xi$, $y(\tau)=\eta$ and by $t\mapsto (x_2(t,\tau,\xi,\eta), y(t,\tau,\eta))$ the solution of (\ref{Heq2}) such that $x(\tau)=\xi$, $y(\tau)=\eta$.

We are now in a position to formulate our first result which gives conditions under which~\eqref{Heq2} and~\eqref{Heq1} are topologically equivalent.
\begin{theorem}\label{t1}
Suppose that  
\begin{equation}\label{bound}
N:=\sup_{t\in \R} \int^{\infty}_{-\infty}|\mathcal G(t,s)|\mu(s)\, ds<\infty  \quad \text{and} \quad q:= \sup_{t\in \R}\int^{\infty}_{-\infty}|\mathcal G(t,s)|\gamma(s)\, ds <1.
\end{equation}
Then, there exists a continuous function $H\colon \R \times X\times Y\to X\times Y$ of the form $H(t, x, y)=(x+h(t,x, y), y)$, where  $\sup_{t,x,y}|h(t,x, y)|<\infty$, such that if $t\mapsto (x(t),y(t))$ is a solution of (\ref{Heq1}),
then $t\mapsto H(t,x(t),y(t))$ is a solution of (\ref{Heq2}). In addition, there exists a continuous function $\bar H \colon \R \times X\times Y\to X\times Y$ of the form $\bar H(t, x, y)=(x+\bar h(t,x, y), y)$, where $\sup_{t,x, y}|\bar h(t,x, y)|<\infty$,  such that if $t\mapsto (x(t),y(t))$ is a solution of (\ref{Heq2}),
then $t\mapsto \bar H(t,x(t),y(t))$ is a solution of (\ref{Heq1}). Moreover, $H$ and $\bar H$ are inverses of
each other, that is,
\[ H(t,\bar H(t,x,y))=(x,y)=\bar H(t,H(t,x,y)),\]
for $t\in \R$ and $(x, y)\in X\times Y$.
\end{theorem}

\begin{proof}
We will first establish the existence of $H$. Let $\mathcal Z$ denote the space of all continuous maps $h\colon \R \times X\times Y\to X$ with the property that 
\[
\lVert h\rVert_\infty :=\sup_{t, x , y}|h(t, x, y)|<\infty.
\]
Then, $(\mathcal Z, \lVert \cdot \rVert)$ is a Banach space. Given $h\in \mathcal Z$, we define 
\[\hat h(\tau,\xi,\eta)
 =\int^{\infty}_{-\infty}\mathcal G(\tau,s)f(s,x_1(s,\tau,\xi)+h(s,x_1(s,\tau,\xi),y(s,\tau,\eta)),y(s,\tau,\eta))\, ds,\]
for $\tau \in \R$ and $(\xi, \eta)\in X\times Y$.
We claim that $\hat h\in \mathcal Z$. Indeed, observe that 
\[
\hat h(\tau,\xi,\eta)=\int^{\infty}_{-\infty}\mathcal G(\tau,s)p(s,\tau,\xi,\eta)\, ds,
\]
where
\[
p(s,\tau,\xi,\eta)=f(s,x_1(s,\tau,\xi)+h(s,x_1(s,\tau,\xi),y(s,\tau,\eta)),y(s,\tau,\eta)).
\]
Then,  it follows from~\eqref{feq} and~\eqref{bound} that 
\[
|\hat h(\tau, \xi, \eta)| \le \int_{-\infty}^\infty |\mathcal G(\tau,s)| \mu(s)\, ds \le N, 
\]
for $\tau \in \R$ and $(\xi, \eta)\in X\times Y$. Now
\[
\hat h(\tau,\xi,\eta)=\int^{\tau}_{-\infty}T(\tau,s)P(s)p(s,\tau,\xi,\eta)\, ds
-\int_{\tau}^{\infty}T(\tau,s)(I-P(s))p(s,\tau,\xi,\eta)\, ds
\]
and so, by direct differentiation,
\[\begin{array}{rl}
 \displaystyle\frac{d}{d\tau}\hat h(\tau,\xi,\eta)
&=A(\tau)\hat h(\tau,\xi,\eta)+P(\tau)p(\tau,\tau,\xi,\eta)+(I-P(\tau))p(\tau,\tau,\xi,\eta)\\ \\

&=A(\tau)\hat h(\tau,\xi,\eta)+p(\tau,\tau,\xi,\eta).\end{array}\]
It follows that $\hat h(\tau,\xi,\eta)$ is continuous in $\tau$, locally uniformly
with respect to $(\xi,\eta)$. Continuity with respect to $(\xi,\eta)$ for each fixed $\tau$ follows from the dominated convergence theorem since $|{\mathcal G}(\tau,s)p(s,\tau,\xi,\eta)|
\le |{\mathcal G}(\tau,s)|\mu(s)$. Hence, $\hat h$ is continuous and 
 $\hat h\in \mathcal Z$. In addition, for $h_1, h_2\in \mathcal Z$ we have (using~\eqref{feq} and~\eqref{bound}) that 
\[ \|\hat h_1-\hat h_2\|_{\infty}\le q\|h_1-h_2\|_\infty.\]
We conclude that the map $T\colon \mathcal Z\to \mathcal Z$ defined by $T(h)=\hat h$, $h\in \mathcal Z$ is a contraction. Therefore, $T$ has a unique fixed point $h\in \mathcal Z$. Then,
\[h(\tau,\xi,\eta)
=\int^{\infty}_{-\infty}\mathcal G(\tau,s)f(s,x_1(s,\tau,\xi)+h(s,x_1(s,\tau,\xi),y(s,\tau,\eta)),y(s,\tau,\eta))\, ds,\]
for $\tau \in \R$ and $(\xi, \eta)\in X\times Y$. By using the identities
\begin{equation}\label{id1} y(t,s,y(s,\tau,\eta))=y(t,\tau,\eta) \end{equation}
and
\begin{equation}\label{id2} x_1(t,s,x_1(s,\tau,\xi))=x_1(t,\tau,\xi),\end{equation}
we have that 
\[\begin{array}{rl}
&h(t,x_1(t,\tau,\xi), y(t,\tau,\eta))\\ \\
&=\displaystyle\int^{\infty}_{-\infty}\mathcal G(t,s)f(s,x_1(s,\tau,\xi)+h(s,x_1(s,\tau,\xi),y(s,\tau,\eta)),y(s,\tau,\eta))\, ds.\end{array}\]
This implies that if $t\mapsto (x(t),y(t))$ is a solution of (\ref{Heq1}), then
\begin{equation}\label{Heq3}h(t,x(t),y(t))
=\int^{\infty}_{-\infty}\mathcal G(t,s)f(s,x(s)+h(s,x(s),y(s)),y(s))\, ds\end{equation}
so that
\[\begin{array}{rl}
h(t,x(t),y(t))
&=\displaystyle\int^{t}_{-\infty}T(t,s)P(s)f(s,x(s)+h(s,x(s),y(s)),y(s))\, ds\\ \\
&-\displaystyle\int_{t}^{\infty}T(t,s)(I-P(s))f(s,x(s)+h(s,x(s),y(s)),y(s))\, ds.\end{array}\]
By direct differentiation, we conclude that
\[ \frac{d}{dt}h(t,x(t),y(t))=A(t)h(t,x(t),y(t))+f(t,x(t)+h(t,x(t),y(t)),y(t)),\] 
and thus $t\mapsto (x(t)+h(t,x(t),y(t)),y(t))$ is a solution of (\ref{Heq2}). For $\tau \in \R$ and $(\xi, \eta)\in X\times Y$, we set
\[ H(\tau,\xi,\eta)=(\xi+h(\tau,\xi,\eta),\eta).\]
From the preceding discussion, we have that if $t\mapsto (x(t),y(t))$ is a solution of (\ref{Heq1}), then
$t\mapsto H(t,x(t),y(t))$ is a solution of (\ref{Heq2}).

We now establish  the existence of $\bar H$. Note the identity
\begin{equation}\label{id3} x_2(t,s,x_2(s,\tau,\xi,\eta),y(s,\tau,\eta))=x_2(t,\tau,\xi,\eta).
\end{equation}
Set 
\begin{equation}\label{90}\bar h(\tau,\xi,\eta)= -\int^{\infty}_{-\infty}\mathcal G(\tau,s)f(s,x_2(s,\tau,\xi,\eta),y(s,\tau,\eta))\, ds.\end{equation}
Similarly to $\hat h$, we can prove that $\bar h\in \mathcal Z$. Using the identities (\ref{id1}) and (\ref{id3}), we have that 
\[\bar h(t,x_2(t,\tau,\xi,\eta),y(t,\tau,\eta))
=-\int^{\infty}_{-\infty}\mathcal G(t,s)f(s,x_2(s,\tau,\xi,\eta),y(s,\tau,\eta))\, ds.\]
Hence, if $t\mapsto (x(t), y(t))$ is a solution of (\ref{Heq2}), we have that 
\begin{equation}\label{Heq4}\bar h(t,x(t),y(t))
= -\int^{\infty}_{-\infty}\mathcal G(t,s)f(s,x(s),y(s))ds.\end{equation}
By direct differentiation,
\[ \frac{d}{dt}\bar h(t,x(t),y(t))=A(t)\bar h(t,x(t),y(t))-f(t,x(t),y(t)),\]
and thus $t\mapsto (x(t)+\bar h(t,x(t),y(t)),y(t))$ is a solution of (\ref{Heq1}). For $\tau \in \R$ and $(\xi, \eta)\in X\times Y$, set
\[ \bar H(\tau,\xi,\eta)=(\xi+\bar h(\tau,\xi,\eta),\eta).\]
By the preceding discussion,  if $t\mapsto (x(t),y(t))$ is a solution of (\ref{Heq2}), then 
$t\mapsto \bar H(t,x(t),y(t))$ is a solution of (\ref{Heq1}).

Now we prove  that $H(t,\bar H(t,x,y))=(x,y)$ for $t\in \R$ and  $(x, y)\in X\times Y$. Fix $\tau \in \R$ and $(\xi, \eta)\in X\times Y$. Let $t\mapsto (x(t),y(t))$ be the solution of (\ref{Heq2}) such that $x(\tau)=\xi$, $y(\tau)=\eta$. Then $t\mapsto (z(t),y(t))=\bar H(t,x(t),y(t))$ is a solution of (\ref{Heq1}) and, using (\ref{Heq4}), \begin{equation}\label{Heq5}z(t)
=x(t)+\bar h(t,x(t),y(t))=x(t)-\displaystyle\int^{\infty}_{-\infty}\mathcal G(t,s)f(s,x(s),y(s))\, ds.\end{equation}
Moreover, $t\mapsto (u(t),y(t))=H(t,z(t),y(t))$ is a solution of (\ref{Heq2}) such that 
\[ (u(\tau),y(\tau))=H(\tau,\bar H(\tau,\xi,\eta)).\]
Then, using (\ref{Heq3}),
\[\begin{array}{rl}
u(t)
&=z(t)+h(t,z(t),y(t))\\ \\
&=z(t)+\displaystyle\int^{\infty}_{-\infty}\mathcal G(t,s)f(s,z(s)+h(s,z(s),y(s)),y(s))\, ds\\ \\
&=z(t)+\displaystyle\int^{\infty}_{-\infty}\mathcal G(t,s)f(s,u(s),y(s))\, ds.
\end{array}
\]
Hence, 
\begin{equation}\label{Heq6} z(t)=u(t)-\int^{\infty}_{-\infty}\mathcal G(t,s)f(s,u(s),y(s))\, ds.
\end{equation}
By comparing (\ref{Heq5}) and (\ref{Heq6}), we obtain that 
\[u(t)-x(t)=\int^{\infty}_{-\infty}\mathcal G(t,s)[f(s,u(s),y(s))-f(s,x(s),y(s))
]\, ds.
\]
Therefore, 
\[ |u(t)-x(t)|
\le \int^{\infty}_{-\infty}|\mathcal G(t,s)|\gamma(s)|u(s)-x(s)|\, ds,\]
which together with~\eqref{bound} implies that 
\[ \|u-x\|_{\infty}\le q\|u-x\|_{\infty}.\]
Since $q<1$, we conclude that $u=x$. So $(u(\tau),y(\tau))=(x(\tau),y(\tau))$, and thus
\[ H(\tau,\bar H(\tau,\xi,\eta))=(\xi,\eta),\]
as required.

In order to complete the proof of the theorem we show that  $\bar H(t,H(t,x,y))=(x, y)$ for $t\in \R$ and  $(x, y)\in X\times Y$. Fix $\tau \in \R$ and 
$(\xi, \eta)\in X\times Y$. Let $t\mapsto (x(t),y(t))$ be the solution of (\ref{Heq1}) such that $x(\tau)=\xi$, $y(\tau)=\eta$. Then 
$t\mapsto (z(t),y(t))=H(t,x(t),y(t))$ is a solution of (\ref{Heq2})
and, using (\ref{Heq3}),
\[\begin{array}{rl}
z(t)
&=x(t)+h(t,x(t),y(t))\\ \\
&=x(t)+\displaystyle\int^{\infty}_{-\infty}\mathcal G(t,s)f(s,x(s)+h(s,x(s),y(s)),y(s)) \, ds.
\end{array}\]
Hence, 
\begin{equation}\label{Heq7} z(t)=x(t)+\displaystyle\int^{\infty}_{-\infty}\mathcal G(t,s)f(s,z(s),y(s))\, ds.
\end{equation}
Furthermore,  $t\mapsto (u(t),y(t))=\bar H(t,z(t),y(t))$ is a solution of (\ref{Heq1}) such that
\[ (u(\tau),y(\tau))=\bar H(\tau,H(\tau,\xi,\eta)).\]
Then, using (\ref{Heq4}),
\[u(t)
=z(t)+\bar h(t,z(t),y(t))
=z(t)-\displaystyle\int^{\infty}_{-\infty}\mathcal G(t,s)f(s,z(s),y(s))\, ds,
\]
and therefore
\begin{equation}\label{Heq8} z(t)=u(t)+\int^{\infty}_{-\infty}\mathcal G(t,s)f(s,z(s),y(s))\, ds.
\end{equation}
By comparing the two expressions (\ref{Heq7}) and (\ref{Heq8}) for $z(t)$, we conclude that $u=x$. Hence $(u(\tau),y(\tau))=(x(\tau),y(\tau))$, and therefore
\[ \bar H(\tau,H(\tau,\xi,\eta))=(\xi,\eta),\]
as required. The proof of the theorem is completed.
\end{proof}

\begin{remark}
In addition to the assumptions in the statement of Theorem~\ref{t1}, let us suppose that there exists $T_0>0$ such that 
\[
A(t+T_0)=A(t), \ g(t+T_0, y)=g(t,y) \ \text{and} \ f(t+T_0, x, y)=f(t, x, y), 
\]
for $t\in \R$, $y\in Y$ and $(x, y)\in X\times Y$ and that ${\mathcal G}(t+T_0,s+T_0)={\mathcal G}(t,s)$ for all $t$ and $s$ (the latter holds
if and only if $P(s+T_0)=P(s)$ for all $s$).  Then, there exist $H$ and $\bar H$ as in the statement of Theorem~\ref{t1} satisfying 
\begin{equation}\label{new0}
H(t+T_0, x, y)=H(t, x, y) \quad \text{and} \quad \bar H(t+T_0, x, y)=\bar H(x, y), 
\end{equation}
for $t\in \R$ and $(x, y)\in X\times Y$. Indeed, this can be proved by slightly adjusting the proof of Theorem~\ref{t1}. Firstly, one can easily show that 
\begin{equation}\label{new1}
(x_1(t+T_0,\tau+T_0,\xi), y(t+T_0,\tau+T_0,\eta))=(x_1(t,\tau,\xi), y(t,\tau,\eta))
\end{equation}
and 
\begin{equation}\label{new2}
 (x_2(t+T_0,\tau+T_0,\xi,\eta), y(t+T_0,\tau+T_0,\eta))= (x_2(t,\tau,\xi,\eta), y(t,\tau,\eta)),
\end{equation}
for $t, \tau \in \R$ and $(\xi,\eta)\in X\times Y$. Let $\mathcal Z$ and $T$ be as in the proof of Theorem~\ref{t1}. Furthermore, let $\mathcal Z^{T_0}$ be the set of all $h\in \mathcal Z$ such that $h(t+T_0, x, y)=h(t,x,y)$ for $t\in \R$ and $(x, y)\in X\times Y$. Then, $\mathcal Z^{T_0}$ is a closed subset of
$\mathcal Z$. Using~\eqref{new1}, it is easy to show that $T(\mathcal Z^{T_0})\subset \mathcal Z^{T_0}$, which yields the first equality in~\eqref{new0}.  
Moreover, \eqref{90} and~\eqref{new2} imply that the second equality in~\eqref{new0} also holds. In particular, in the the autonomous case,
the functions $H$ and $\bar H$ are independent of $t$ provided that $P(s)$ is constant.

\end{remark}

\begin{remark}
The proof of Theorem~\ref{t1} is inspired by the work of Reinfelds and Steinberga~\cite{RS}. In particular, the main result in~\cite{RS} is  applicable to the problem of  the topological equivalence between systems
\[
\dot x=A(t)x+f(t,x)
\]
and
\[
\dot x=A(t)x.
\]
Under the same assumptions as in the present paper (see~\eqref{bound}), it follows from~\cite[Theorem 3]{RS} that the  above systems are topologically equivalent. 

Besides considering the more general systems~\eqref{Heq2} and~\eqref{Heq1}, we also provide some additional details (in comparison to the 
proof of~\cite[Theorem 3]{RS}). Namely, in~\cite{RS} it was not explicitly proved that the conjugacies $H$ and $\bar H$ are inverses of each other.
\end{remark}

\begin{remark}
We stress that the result analogous to Theorem~\ref{t1} was established by Shi and Xiong~\cite{SX} (see also~\cite{XWKR}) under the assumption that $\dot x=A(t)x$ admits an exponential dichotomy (see Subsection~\ref{EX}). In addition, it is assumed in~\cite{SX} that $\dot y=g(t,y)$ is a linear system. 
\end{remark}
The following example (essentially taken from~\cite{RS}) shows that Theorem~\ref{t1} is much more general from the above described result of Shi and Xiong~\cite{SX}.
\begin{example}\label{exp}
Let $X=\R^3$ and for $t\in \R$, set
\[
A(t)=\begin{pmatrix}
0 & -1 & 0 \\
1 & 0 & 0 \\
0 & 0 & \frac{-2t}{1+t^2} 
\end{pmatrix} \quad  \text{and}  \quad
P(t)=\begin{pmatrix}
0 & 0 &0 \\
0 & 0 & 0 \\
0 &0 & 1
\end{pmatrix}.
\]
One can easily show (see~\cite[Section 4]{RS}) that~\eqref{bound} holds  whenever 
\[
\int_{-\infty}^\infty (1+s^2) \mu(s) \, ds<+\infty \quad \text{and} \quad \int_{-\infty}^\infty (1+s^2)\gamma (s)\, ds <1.
\]
Moreover, as observed in~\cite{RS}, $\dot x=A(t)x$ does not admit an exponential dichotomy (or even the so-called ordinary dichotomy~\cite[Chapter 2]{C}).
\end{example}

In the subsections below on H\" older continuity, we assume that $A(t)$, $f(t,x,y)$ and $g(t,y)$ satisfy the same conditions
as assumed in Theorem \ref{t1} and that ${\mathcal G}(t,s)$ is as defined as in (\ref{G}). However other conditions may now be added. 

\subsection{H\" older continuity of $H$ and $\bar H$ in $x$}\label{hcx}
We suppose there exist a continuous function $\Delta_1(t,s)>0$ such that for all $t$ and $s$,
\begin{equation}\label{T} |T(t,s)|\le \Delta_1(t,s)\end{equation}
and a continuous function $\Delta_2(t,s)>0$ such that
if $t\mapsto (x(t),y(t))$ and $t\mapsto (z(t),y(t))$ are solutions of (\ref{Heq2}),
then for all $t$ and $s$,
\begin{equation}\label{oo}  |x(t)-z(t)|\le \Delta_2(t,s)|x(s)-z(s)|.\end{equation}
Furthermore, we assume there exists $M\ge 1$ and a continuous function $\varepsilon \colon \R \to (0, \infty)$  such that 
\begin{equation}\label{a1} |f(t,x,y)|\le M\end{equation}
and 
\begin{equation}\label{a2} |f(t,x_1,y)-f(t,x_2,y)| \le \varepsilon (t)|x_1-x_2|,\end{equation}
for $t\in \R$, $x, x_1, x_2\in X$ and $y\in Y$. In addition, we assume that there exists $N\ge 1$  such that \begin{equation}\label{ap} \varepsilon(t)\le N.\end{equation} Finally, we suppose that
\begin{equation}\label{r} \sup_{t\in \R} \int^{\infty}_{-\infty}|\mathcal G(t,s)|ds<\infty, \quad 
q:=\sup_{t\in \R}   \int^{\infty}_{-\infty}|\mathcal G(t,s)|\varepsilon(s)ds<1.\end{equation}
Observe that the above condition implies that~\eqref{bound} holds and thus Theorem~\ref{t1} is valid. Hence,  there is a function $H(t,x,y)=(x+h(t,x,y),y)$ sending the solutions of
(\ref{Heq1}) onto the solutions of (\ref{Heq2}) and a function $\bar H(t,x,y)=(x+\bar h(t,x,y),y)$ sending the solutions of
(\ref{Heq2}) onto the solutions of (\ref{Heq1}).

\begin{theorem}\label{t2}
 Let $C>0$ and $0<\alpha<1$ be given. Then if 
\begin{equation}\label{c1} \max\{2M,N\}(1+C)\sup_{t\in \R}\int^{\infty}_{-\infty}|\mathcal G(t,s)|\varepsilon^{\alpha}(s)\Delta^{\alpha}_1(s,t)\, ds
\le C,\end{equation}
we have
\[ |h(t,x_1,y)-h(t,x_2,y)|\le C|x_1-x_2|^{\alpha}, \quad \text{for $t\in \R$, $x_1, x_2\in X$ and $y\in Y$.}\]
Moreover, if
\begin{equation}\label{c2} 2M \sup_{t\in \R} \int^{\infty}_{-\infty}|\mathcal G(t,s)|\varepsilon^{\alpha}(s)\Delta^{\alpha}_2(s,t)\, ds
\le C,\end{equation}
then
\[ |\bar h(t,x_1,y)-\bar h(t,x_2,y)|\le C|x_1-x_2|^{\alpha}, \quad \text{for $t\in \R$, $x_1, x_2\in X$ and $y\in Y$.}\]
\end{theorem}

\begin{proof}
We begin by observing that~\eqref{a1} and~\eqref{a2} imply that 
\begin{equation}\label{Heq9} \begin{array}{rl}
|f(t,x,y)-f(t,z,y)|
&=|f(t,x,y)-f(t,z,y)|^{1-\alpha}|f(t,x,y)-f(t,z,y)|^{\alpha}\\ \\
&\le 2M\varepsilon^{\alpha}(t)|x-z|^{\alpha},
\end{array}\end{equation}
for $t\in \R$, $x, z\in X$ and $y\in Y$. 

Let $\mathcal Z$ be as in the proof of Theorem~\ref{t1}. Furthermore, let $\mathcal Z'$ denote the set of all $\psi \in \mathcal Z$  such that
\[
|\psi(t,x_1, y)-\psi(t,x_2, y)|\le C|x_1-x_2|^\alpha, \quad \text{for $t\in \R$, $x_1, x_2\in X$ and $y\in Y$.}
\]
Then, $\mathcal Z'$ is a closed subset of $\mathcal Z$. We now prove that $T(\mathcal Z')\subset \mathcal Z'$, where $T$ is as in the proof of Theorem~\ref{t1}. We recall that 
\begin{equation}\label{opT}
(T\psi)(\tau, \xi, \eta)= \int^{\infty}_{-\infty}\mathcal G(\tau,s)f(s,x_1(s,\tau,\xi)+\psi (s,x_1(s,\tau,\xi),y(s,\tau,\eta)),y(s,\tau,\eta))\, ds,
\end{equation}
for $\tau \in \R$ and $(\xi, \eta) \in X \times Y$.
 Take an arbitrary $\psi \in \mathcal Z'$. We observe  (using~\eqref{a2}, \eqref{ap} and~\eqref{Heq9})  that 
\begin{equation}\label{tat}
\begin{split}
&|f(t,x_1+\psi(t,x_1,y),y)-f(t,x_2+\psi(t,x_2,y),y) | \\
&\le \min\{\varepsilon(t)[|x_1-x_2|+|\psi(t,x_1,y)-\psi(t,x_2,y)|],\\ 
&\qquad 2M\varepsilon^{\alpha}(t)[|x_1-x_2|+|\psi(t,x_1,y)-\psi(t,x_2,y)|]^{\alpha}\} \\
&\le M_1\varepsilon^{\alpha}(t)\min\{|x_1-x_2|+C|x_1-x_2|^{\alpha},  [|x_1-x_2|+C|x_1-x_2|^{\alpha}]^{\alpha}\}\\ 
&\le  M_1\varepsilon^{\alpha}(t) \begin{cases}
(1+C)^\alpha|x_1-x_2|^{\alpha} & \text{if $|x_1-x_2|>1$; (taking the right one)} \\
(1+C)|x_1-x_2|^{\alpha} & \text{if $|x_1-x_2|\le 1$ (taking the left one)}
\end{cases}\\
&\le M_1\varepsilon^{\alpha}(t)(1+C)|x_1-x_2|^{\alpha},
\end{split}
\end{equation}
where $M_1=\max\{N,2M\}$. Now, 
\[ (T\psi)(t,\xi_1,\eta)-(T\psi)(t,\xi_2,\eta)=\int^{\infty}_{-\infty}\mathcal G(t,s)p(s)\, ds,\]
where
\[\begin{split} p(s)
&= f(s,x_1(s,t,\xi_1)+\psi(s,x_1(s,t,\xi_1),y(s,t,\eta)),y(s,t,\eta))\\ 
&\phantom{=} -f(s,x_1(s,t,\xi_2)+\psi(s,x_1(s,t,\xi_2),y(s,t,\eta)),y(s,t,\eta)).
\end{split}\]
Using~\eqref{tat}, we see that 
\[ \begin{split}
|p(s)|
&\le M_1\varepsilon^{\alpha}(s)(1+C)|x_1(s,t,\xi_2)-x_1(s,t,\xi_1)|^{\alpha}\ \\ 
&\le M_1\varepsilon^{\alpha}(s)(1+C)[\Delta_1(s,t)|\xi_1-\xi_2|]^{\alpha}\\ 
&= M_1\varepsilon^{\alpha}(s)(1+C)\Delta^{\alpha}_1(s,t)|\xi_1-\xi_2|^{\alpha}.
\end{split}\]
Therefore, by~\eqref{c1} we have that 
\[\begin{split}
& |(T\psi)(t,\xi_1,\eta)-(T\psi)(t,\xi_2,\eta)|\\ 
&\le \displaystyle M_1(1+C)|\xi_1-\xi_2|^{\alpha} \sup_{t\in \R}\int^{\infty}_{-\infty}|\mathcal G(t,s)|\varepsilon^{\alpha}(s)\Delta^{\alpha}_1(s,t)\, ds
\\ 
&\le C|\xi_1-\xi_2|^{\alpha},
\end{split}\]
for $t\in \R$, $\xi_1, \xi_2\in X$ and $\eta \in Y$. Therefore, $T\psi \in \mathcal Z'$. Consequently, the unique fixed point $h$ of $T$ belongs to $\mathcal Z'$, which implies the first assertion of the theorem.

In order to establish the second assertion, we recall (see the proof of Theorem~\ref{t1}) that 
\[\bar h(t,\xi,\eta)= -\int^{\infty}_{-\infty}\mathcal G(t,s)f(s,x_2(s,t,\xi,\eta),y(s,t,\eta))\, ds.
\]
Then
\[ \bar h(t,\xi_1,\eta)-\bar h(t,\xi_2,\eta)
=\int^{\infty}_{-\infty}\mathcal G(t,s)p(s)ds,\]
where
\[ p(s)=f(s,x_2(s,t,\xi_2,\eta),y(s,t,\eta))-f(s,x_2(s,t,\xi_1,\eta),y(s,t,\eta)).\]
By~\eqref{Heq9}, we have that 
\[ |p(s)|\le 2M\varepsilon^{\alpha}(s)|x_2(s,t,\xi_2,\eta)-x_2(s,t,\xi_1,\eta)|^{\alpha} 
\le 2M\varepsilon^{\alpha}(s)\Delta_2^{\alpha}(s,t)|\xi_1-\xi_2|^{\alpha}.\]
Consequently, using~\eqref{c2} we conclude that 
\[\begin{split}
& |\bar h(t,\xi_1,\eta)-\bar h(t,\xi_2,\eta)|\\ 
&\le \displaystyle 2M |\xi_1-\xi_2|^{\alpha}\int^{\infty}_{-\infty}|\mathcal G(t,s)|\varepsilon^{\alpha}(s)\Delta^{\alpha}_2(s,t)\, ds
\\ 
&\le C|\xi_1-\xi_2|^{\alpha},
\end{split}\]
for $t\in \R$, $\xi_1, \xi_2\in X$ and $\eta\in Y$. The proof of the theorem is completed.
\end{proof}

\begin{remark}\label{rm1}
Let us equip $X\times Y$ with the norm $|(x,y)|=|x|+|y|$, for $(x, y)\in X\times Y$.
We observe that under~\eqref{c1}, $H$ is H\" older continuous in $x$ on every bounded subset of $X$. Moreover, \eqref{c2} implies that $\bar H$ is H\" older continuous in $x$ on every bounded subset of $X$. Indeed, assume that~\eqref{c1} holds and that $\tilde X\subset X$ is bounded. 
Then, we have that 
\[
\begin{split}
|H(t,x_1,y)-H(t,x_2,y)| &=|(x_1+h(t,x_1, y))-(x_2+h(t,x_2, y))| \\
&\le |x_1-x_2|+|h(t,x_1, y)-h(t,x_2, y)| \\
&\le |x_1-x_2|+C|x_1-x_2|^\alpha \\
&=(|x_1-x_2|^{1-\alpha}+C)|x_1-x_2|^\alpha \\
&\le C' |x_1-x_2|^\alpha,
\end{split}
\]
for $t\in \R$, $x_1, x_2\in \tilde X$ and $y\in Y$, where
\[
C'=C+\sup_{x_1, x_2\in \tilde X} |x_1-x_2|^{1-\alpha}>0.
\]
The same argument applies for $\bar H$.
\end{remark}

\subsection{H\" older continuity of $H$ and $\bar H$ in $y$} We suppose that if  $t\mapsto y(t)$ and $t\mapsto w(t)$ are solutions of $\dot y=g(t,y)$, then 
\begin{equation}\label{sig} |y(t)-w(t)|\le \sigma(t,s)|y(s)-w(s)|,\end{equation}
for some continuous function $\sigma(t,s)>0$. In addition, we assume that there exists a continuous function $\Delta_3(t,s)>0$  such that if $t\mapsto (x(t),y(t))$
and $t\mapsto (z(t),w(t))$ are solutions of (\ref{Heq2}) with $x(s)=z(s)$, then for all $t$ and $s$
\begin{equation}\label{Del} |x(t)-z(t)|+|y(t)-w(t)|\le \Delta_3(t,s)|y(s)-w(s)|.\end{equation}
We continue to assume that~\eqref{a1} holds with $M\ge 1$.  Moreover, we assume that there exists a continuous function
$\varepsilon \colon \R \to (0, \infty)$ satisfying~\eqref{ap} (with some $N\ge 1$) and such that 
\begin{equation}\label{113} |f(t,x_1,y_1)-f(t,x_1,y_2)|\le \varepsilon(t)[|x_1-x_2|+|y_1-y_2|],\end{equation}
for $t\in \R$, $x_1, x_2\in X$ and $y_1, y_2\in Y$.
 Finally, suppose that~\eqref{r} holds. Hence, Theorem~\ref{t1} is again applicable and consequently there exist $H$ and $\bar H$ as in the statement of that result.
\begin{theorem}\label{t3}
 Let $C>0$ and $0<\alpha<1$ be given. Then, if 
\begin{equation}\label{c3}
\max\{2M,N\}(1+C)\sup_{t\in \R}\int^{\infty}_{-\infty}|\mathcal G(t,s)|
\varepsilon^{\alpha}(s)\sigma^{\alpha}(s,t)\, ds
	\le C,
\end{equation}
we have that
	\[ |h(t,x,y_1)-h(t,x,y_2)|\le C|y_1-y_2|^{\alpha}, \quad \text{for $t\in \R$, $x\in X$ and $y_1, y_2\in Y$.}\]
Moreover, if \begin{equation}\label{c44} 2M\sup_{t\in \R} \int^{\infty}_{-\infty}|
\mathcal G(t,s)|\varepsilon^{\alpha}(s)\Delta^{\alpha}_3(s,t)\, ds
	\le C,\end{equation}
then 
\[ |\bar h(t,x,y_1)-\bar h(t,x,y_2)|\le C|y_1-y_2|^{\alpha}, \quad \text{for $t\in \R$, $x\in X$ and $y_1, y_2\in Y$.}\]
\end{theorem}

\begin{proof}
We begin by observing that~\eqref{a1} and~\eqref{113} imply that 
\begin{equation}\label{Heq11} \begin{array}{rl}
&|f(t,x_1,y_1)-f(t,x_2,y_2)|\\ \\
&=|f(t,x_1,y_1)-f(t,x_2,y_2)|^{1-\alpha}|f(t,x_1,y_1)-f(t,x_2,y_2)|^{\alpha}\\ \\
&\le 2M\varepsilon^{\alpha}(t)[|x_1-x_2|+|y_1-y_2|]^{\alpha},
\end{array}\end{equation}
for $t\in \R$, $x_1, x_2\in X$ and $y_1, y_2\in Y$. 

Let $\mathcal Z''$ denote the set of all $\psi \in \mathcal Z$ (where $\mathcal Z$ is again as in the proof of Theorem~\ref{t1}) such that 
\[ |\psi(t,x,y_1)-\psi(t,x,y_2)|\le C|y_1-y_2|^{\alpha}, \quad \text{for $t\in \R$, $x\in X$ and $y_1, y_2\in Y$.}\]
Then, $\mathcal Z''$ is a closed subset of $\mathcal Z$. We now prove that $T(\mathcal Z'')\subset \mathcal Z''$, where $T$ is as in the proof of Theorem~\ref{t1} (see~\eqref{opT}). Take an arbitrary $\psi\in \mathcal Z''$. By~\eqref{ap}, \eqref{113} and~\eqref{Heq11}, we have that 
\begin{equation}\label{op}
\begin{split}
&|f(t,x+\psi(t,x,y_1),y_1)-f(t,x+\psi(t,x,y_2),y_2) |\\ 
&\le \min\{\varepsilon(t)[|\psi(t,x,y_1)-\psi(t,x,y_2)|+|y_1-y_2|],\\
&\phantom{\le} \qquad 2M\varepsilon^{\alpha}(t)[|\psi(t,x,y_1)-\psi(t,x,y_2)|+|y_1-y_2|]^{\alpha}\}\\ 
&\le M_1\varepsilon^{\alpha}(t)\min\{|y_1-y_2|+C|y_1-y_2|^{\alpha},  [|y_1-y_2|+C|y_1-y_2|^{\alpha}]^{\alpha}\}\\ 
&\le  M_1\varepsilon^{\alpha}(t) \begin{cases}
(1+C)^\alpha|y_1-y_2|^{\alpha} & \text{if $|y_1-y_2|>1$; (taking the one on the right side)} \\
(1+C)|y_1-y_2|^{\alpha} & \text{if $|y_1-y_2|\le 1$ (taking the one on the left side)}
\end{cases}\\
&\le M_1\varepsilon^{\alpha}(t)(1+C)|y_1-y_2|^{\alpha},
\end{split}
\end{equation}
where $M_1=\max \{2M, N\}$. Then,
\[ (T\psi)(t,\xi,\eta_1)-(T\psi)(t,\xi,\eta_2)=\int^{\infty}_{-\infty}\mathcal G(t,s)p(s)\, ds,\]
where
\[\begin{split}
p(s)&= f(s,x_1(s,t,\xi)+\psi(s,x_1(s,t,\xi),y(s,t,\eta_1)),y(s,t,\eta_1))\\ 
&\phantom{=} -f(s,x_1(s,t,\xi)+\psi(s,x_1(s,t,\xi),y(s,t,\eta_2)),y(s,t,\eta_2)).
\end{split}\]
By~\eqref{op}, we see that 
\[ \begin{split}
|p(s)|
&\le M_1\varepsilon^{\alpha}(s)(1+C)|y(s,t,\eta_1)-y(s,t,\eta_2)|^{\alpha}\\ 
&\le M_1\varepsilon^{\alpha}(s)(1+C)\sigma^{\alpha}(s,t)|\eta_1-\eta_2|^{\alpha}.
\end{split}\]
Hence, \eqref{c3} implies that 
\[\begin{split}
& |(T\psi)(t,\xi,\eta_1)-(T\psi)(t,\xi,\eta_2)|\\ 
&\le \displaystyle M_1(1+C) |\eta_1-\eta_2|^{\alpha} \sup_{t\in \R}\int^{\infty}_{-\infty}|\mathcal G(t,s)|\varepsilon^{\alpha}(s)\sigma^{\alpha}(s,t)\, ds
\\
&\le C|\eta_1-\eta_2|^{\alpha},
\end{split}\]
for $t\in \R$, $\xi \in X$ and $\eta_1, \eta_2\in Y$. Therefore, $T\psi \in \mathcal Z''$. Consequently, the unique fixed point $h$ of $T$ belongs to $\mathcal Z''$, which implies the first assertion of the theorem.

On the other hand, we recall from Theorem~\ref{t1} that $\bar h$ is given by
\[\bar h(t,\xi,\eta)=-\int^{\infty}_{-\infty}\mathcal G(t,s)f(s,x_2(s,t,\xi,\eta),y(s,t,\eta))\, ds.
\]
Thus,
\[ \bar h(t,\xi,\eta_1)-\bar h(t,\xi,\eta_2)=\int^{\infty}_{-\infty}\mathcal G(t,s)p(s)\, ds,\]
where
\[p(s)= f(s,x_2(s,t,\xi,\eta_2),y(s,t,\eta_2))-f(s,x_2(s,t,\xi,\eta_1),y(s,t,\eta_1)).
\]
Then, using (\ref{Heq11}) we have that 
\[ \begin{split}
|p(s)|
&\le 2M\varepsilon^{\alpha}(s)[|x_2(s,t,\xi,\eta_1)-x_2(s,t,\xi,\eta_2)|+|y(s,t,\eta_1)-y(s,t,\eta_2)|]^{\alpha}\\ 
&\le 2M\varepsilon^{\alpha}(s)[\Delta_3(s,t)|\eta_1-\eta_2|]^{\alpha}\\ 
&= 2M\varepsilon^{\alpha}(s)\Delta^{\alpha}_3(s,t)|\eta_1-\eta_2|^{\alpha}.
\end{split}\]
Therefore, \eqref{c44} implies that 
\[\begin{split}
& |\bar h (t,\xi,\eta_1)-\bar h(t,\xi,\eta_2)|\\ 
&\le \displaystyle 2M |\eta_1-\eta_2|^{\alpha}\sup_{t\in \R}\int^{\infty}_{-\infty}|
\mathcal G(t,s)|\varepsilon^{\alpha}(s)\Delta^{\alpha}_3(s,t)\, ds
\\ 
&\le C|\eta_1-\eta_2|^{\alpha},
\end{split}\]
which establishes the second assertion of the theorem. 
\end{proof}

\begin{remark}\label{rm2}
As in Remark~\ref{rm1}, we have that~\eqref{c3} implies that $H$ is H\" older continuous in $y$ on each bounded subset of $Y$ and that under~\eqref{c44}, $\bar H$ is H\" older continuous in $y$ on each bounded subset on $Y$.
\end{remark}

\begin{remark}\label{rm3}
Moreover, \eqref{c1} and~\eqref{c3} imply that $H$ is H\" older continuous in $(x,y)$ on each bounded subset of $X\times Y$. Similarly, \eqref{c2} and~\eqref{c44} imply the same for $\bar H$.  Indeed, take $Q\subset X\times Y$ bounded. Then,
\[
\begin{split}
&|H(t,x_1, y_1)-H(t,x_2,y_2)| \\
 &=|(x_1+h(t,x_1, y_1), y_1)-(x_2+h(t,x_2, y_2), y_2)| \\
&\le |x_1-x_2|+|h(t,x_1, y_1)-h(t,x_2, y_2)|+|y_1-y_2| \\
&\le |x_1-x_2|+|y_1-y_2|+|h(t,x_1, y_1)-h(t,x_2, y_1)|+|h(t,x_2, y_1)-h(t,x_2, y_2)| \\
&\le |x_1-x_2|+|y_1-y_2|+C|x_1-x_2|^\alpha +C|y_1-y_2|^\alpha \\
&=(|x_1-x_2|^{1-\alpha}+C)|x_1-x_2|^\alpha +(|y_1-y_2|^{1-\alpha}+C)|y_1-y_2|^\alpha  \\
&\le C' |(x_1,y_1)-(x_2,y_2)|^\alpha, 
\end{split}
\]
for $t\in \R$, $(x_1,y_1), (x_2, y_2)\in Q$, where 
\[
C'=2C+\sup_{(x_1, y_1), (x_2, y_2)\in Q} (|x_1-x_2|^{1-\alpha}+|y_1-y_2|^{1-\alpha}).
\]
The same argument applies for $\bar H$.
\end{remark}

\subsection{An example}\label{EX}
Let us now discuss a particular example to which Theorems~\ref{t1}, \ref{t2} and~\ref{t3} are applicable. 

First we recall a definition. We assume that $X=\R^d$. We say that  $\dot x=A(t)x$ admits
an {\it exponential dichotomy} with projections $P(t)$, $t\in \R$ if  
\[
T(t,s)P(s)=P(t)T(t,s) \quad \text{for $t, s\in \R$,}
\] 
and there exist positive constants $D_1$, $D_2$, $\lambda_1$, $\lambda_2$ such that
\[
|T(t,s)P(s)|\le D_1e^{-\lambda_1(t-s)} \quad \text{and} \quad |T(s,t)(I-P(t))|\le D_2e^{-\lambda_2(t-s)},
\]
for $t\ge s$. 


In the corollaries below, we assume that $\dot x=A(t)x$ satisfies an exponential 
dichotomy and also some bounded growth and decay conditions. Following the
corollaries, we describe conditions on the Sacker-Sell spectrum which ensure these
conditions hold. Note the exponential dichotomy condition leads to the condition
on the Green's function; the bounded growth and decay conditions lead to the conditions
on the $\Delta_i(t,s)$ functions. In the first corollary, we give conditions ensuring
the existence of a linearization and H\" older continuity in $x$.

\begin{corollary}\label{KOR}
Assume that the following conditions hold:
\begin{enumerate}
\item There exist positive constants $K_1$, $K_2$, $D_1$, $D_2$, $a_2\ge\lambda_1$, $a_1\ge\lambda_2$
 and projections $P(t)$, $t\in \R$ such that for all $t$ and $s$
  \[ T(t,s)P(s)=P(t)T(t,s)\]
and for $t\ge s$
\begin{equation}\label{est-1}\begin{array}{rl}
|T(t,s)| &\le K_1e^{a_1(t-s)},\quad  |T(s,t)| \le K_2e^{a_2(t-s)},\\ \\
|T(t,s)P(s)| & \le D_1e^{-\lambda_1(t-s)},\quad |T(s,t)(I-P(t))|  \le D_2e^{-\lambda_2(t-s)}.
\end{array} \end{equation}
\item  there exists $M\ge 1$ such that~\eqref{a1} holds;
\item there exists $\varepsilon >0$ such that 
\begin{equation}\label{EPs}|f(t,x_1,y)-f(t,x_2,y)|\le \varepsilon|x_1-x_2|, \quad \text{for $t\in \R$, $x_1, x_2\in X$ and $y\in Y$.}\end{equation}
\end{enumerate}
Then, if
\begin{equation}\label{epcon} 
\left(\frac{D_1}{\lambda_1}+\frac{D_2}{\lambda_2}\right)\varepsilon<1,\end{equation}
there exist functions $H$ and $\bar H$  as in the statement of Theorem~\ref{t1}.
Furthermore,  suppose that
\[ 0<\alpha<\min\{\lambda_1/a_2, \lambda_2/a_1\}.\]
(a) Then, given $C>0$, provided
$\varepsilon$ is sufficiently  small so  that
\begin{equation}\label{epcon1} \max\{2M,\varepsilon\}
(1+C)\varepsilon^{\alpha}\displaystyle\left[\frac{D_1K^{\alpha}_1}{\lambda_1-\alpha a_2}
+\frac{D_2K^{\alpha}_2}{\lambda_2-\alpha a_1}\right]
\le C,\end{equation}
we have that 
\begin{equation}\label{conc1} |h(t,x_1,y)-h(t,x_2,y)|\le  C|x_1-x_2|^{\alpha}, \quad \text{for $t\in \R$, $x_1, x_2\in X$ and $y\in Y$.}\end{equation}
(b) Moreover, if  $\varepsilon$ is sufficiently  small so that
\[ 0<\alpha<\min\{\lambda_1/(a_2+K_1\varepsilon), \lambda_2/(a_1+K_2\varepsilon)\},\]
and 
\begin{equation}\label{epcon4} 2M\varepsilon^{\alpha}
\left[\frac{D_1K^{\alpha}_1}{\lambda_1-\alpha(a_2+K_1\varepsilon)}
+\frac{D_2K^{\alpha}_2}{\lambda_2-\alpha(a_1+K_2\varepsilon)}\right] \le C,\end{equation}
then
\begin{equation}\label{conc2} |\bar h(t,x_1,y)-\bar h(t,x_2,y)|\le  C|x_1-x_2|^{\alpha}, \quad \text{for $t\in \R$, $x_1, x_2\in X$ and $y\in Y$.}\end{equation}
\end{corollary}

\begin{proof}
It follows from (\ref{G}) and (\ref{est-1}) that 
\[ |{\mathcal G}(t,s)|\le \begin{cases}D_1e^{-\lambda_1(t-s)}& (t\ge s)\\D_2e^{-\lambda_2(s-t)}& (t<s).\end{cases}\] 
Then 
\[
\sup_{t\in \R}\int^{\infty}_{-\infty}|\mathcal G(t,s)|\, ds \le \left(\frac{D_1}{\lambda_1}+\frac{D_2}{\lambda_2}\right)<\infty
\]
and (\ref{epcon}) implies that
\[
\varepsilon \sup_{t\in \R} \int^{\infty}_{-\infty}|\mathcal G(t,s)|\, ds<1.
\]
We can now apply Theorem~\ref{t1} to conclude that there exist functions $H$ and $\bar H$ as in the statement of that result.

Moreover, observe that~\eqref{est-1} implies that~\eqref{T} holds with 
\[ \Delta_1(t,s)=\begin{cases} 
K_1e^{a_1|t-s|} & \text{for  $t\ge s$;} \\
	K_2e^{a_2|t-s|} & \text{for $t< s$}.
\end{cases}
\]
Then,
\[ \begin{split}
&\int^{\infty}_{-\infty}|\mathcal G(t,s)|\varepsilon^{\alpha}\Delta^{\alpha}_1(s,t)\, ds
\\ 
&\le \displaystyle\int^{t}_{-\infty}D_1e^{-\lambda_1(t-s)}\varepsilon^{\alpha}K^{\alpha}_1
e^{\alpha a_2(t-s)}\, ds
+\int^{\infty}_{t}D_2e^{-\lambda_2(s-t)}\varepsilon^{\alpha}K^{\alpha}_2e^{\alpha a_1(s-t)}\, ds\\ 
&\le \varepsilon^{\alpha}\displaystyle\left[\frac{D_1K^{\alpha}_1}{\lambda_1-\alpha a_2}
+\frac{D_2K^{\alpha}_2}{\lambda_2-\alpha a_1}\right],
\end{split}\]
for  each $t\in \R$. Now Theorem~\ref{t2} and~\eqref{epcon1} imply that~\eqref{conc1} holds.

Next it is proved in Appendix~\ref{A2} that~\eqref{oo} holds with 
\[ \Delta_2(t,s)=\begin{cases}
K_1e^{(a_1+K_1\varepsilon)(t-s)} & \text{for $t\ge s$;} \\
  K_2e^{(a_2+K_2\varepsilon)(s-t)}& \text{for $s\ge t$.}
\end{cases}\]
Then,
\[\begin{split}& \int^{\infty}_{-\infty}|\mathcal G(t,s)|\varepsilon^{\alpha}\Delta^{\alpha}_2(s,t)\, ds
\\ 
&\le \displaystyle\int^{t}_{-\infty}D_1e^{-\lambda_1(t-s)}\varepsilon^{\alpha}
K^{\alpha}_1e^{\alpha(a_2+K_2\varepsilon)(t-s)}\, ds
+\int^{\infty}_{t}D_2e^{-\lambda_2(s-t)}\varepsilon^{\alpha}K^{\alpha}_2
e^{\alpha(a_1+K_1\varepsilon)(s-t)}\, ds,\end{split}\]
which yields that 
\[ \int^{\infty}_{-\infty}|\mathcal G(t,s)|\varepsilon^{\alpha}\Delta^{\alpha}_2(s,t)\, ds
\le \varepsilon^{\alpha}\left[\frac{D_1K^{\alpha}_1}{\lambda_1-\alpha(a_2+K_2\varepsilon)}
+\frac{D_2K^{\alpha}_2}{\lambda_2-\alpha(a_1+K_1\varepsilon)}\right],\]
for $t\in \R$. We conclude from Theorem~\ref{t2} and~\eqref{epcon4} that~\eqref{conc2} holds. The proof of the corollary is completed. 
\end{proof}

In the next corollary we give conditions ensuring
the existence of a linearization and H\" older continuity in $y$.

\begin{corollary}
Assume that conditions (1) and (2) as in Corollary \ref{KOR} and that, in addition, the following conditions hold:
\begin{enumerate}
\item[(3)] there exists $\varepsilon>0$ such that 
\begin{equation}\label{3:47}
|f(t,x_1,y_1)-f(t,x_2,y_2)|\le \varepsilon[|x_1-x_2|+|y_1-y_2|],
\end{equation}
for $t\in \R$, $x_1, x_2\in X$ and $y_1, y_2\in Y$;
\item[(4)] there exists $M_2>0$ such that 
\begin{equation}|g(t,y)-g(t,w)|\le M_2|y-w|, \quad \text{for $t\in \R$ and $y, w\in Y$.}\end{equation}
\end{enumerate}
Then, if (\ref{epcon}) holds, there exist functions $H$ and $\bar H$  as in the 
statement of Theorem~\ref{t1}.
Moreover,  we have the following:
\begin{itemize}
\item[(a)]  suppose that 
\[ 0<\alpha<\min\{\lambda_1/M_2, \lambda_2/M_2\}. \]
Given $C>0$, provided $\varepsilon$ is sufficiently small so that 
\begin{equation}\label{cor2cond1} 2M(1+C)\varepsilon^{\alpha}\left[\frac{D_1}{\lambda_1-\alpha M_2}+\frac{D_2}{\lambda_2-\alpha M_2}\right]
\le C,\end{equation}
we have that 
\begin{equation}\label{Con}
 |h(t,x,y_1)-h(t,x,y_2)|\le  C|y_1-y_2|^{\alpha}, 
\end{equation}
for $t\in \R$, $x\in X$ and $y_1, y_2\in Y$.
\item[(b)]  suppose
\[ 0<\alpha<\min\{\lambda_1/M_3,\lambda_2/M_3\},\]
where $M_3=\max\{M_2,a_1,a_2\}$. Then given $C>0$, provided
$\varepsilon >0$ is sufficiently  small so that
\[ 0<\alpha<\min\{\lambda_1/(M_3+K_2\varepsilon),\lambda_2/M_3+K_1\varepsilon)\},\]
and 
\begin{equation}\label{cor2condition3} 2^{1+\alpha} M\varepsilon^{\alpha}
\left[\frac{D_1}{\lambda_1-\alpha(M_3+K_2\varepsilon)}
+\frac{D_2}{\lambda_2-\alpha (M_3+K_1\varepsilon)}\right]\le C,\end{equation}
we have that 
\begin{equation}\label{Con2} |\bar h(t,x,y_1)-\bar h(t,x,y_2)|\le  C|y_1-y_2|^{\alpha},\end{equation}
for $t\in \R$, $x\in X$ and $y_1,y_2\in Y$.
\end{itemize}
\end{corollary}

\begin{proof}
The existence of $H$ and $\bar H$ follows as in the proof of Corollary~\ref{KOR}. Let us now establish~(a). By a simple Gronwall lemma argument, we have that~\eqref{sig} holds with
\begin{equation}\label{M2} \sigma(t,s)=e^{M_2|t-s|}.\end{equation}
 Then, 
\[\begin{split}
&\int^{\infty}_{-\infty}|\mathcal G(t,s)|\sigma^{\alpha}(s,t)\, ds\\ \\
&\le  \int^{t}_{-\infty}D_1e^{-\lambda_1(t-s)}e^{\alpha M_2(t-s)}\, ds
+\int^{\infty}_{t}D_2e^{-\lambda_2(s-t)}e^{\alpha M_2(s-t)}\, ds\\ \\
&\le   \left[\frac{D_1}{\lambda_1-\alpha M_2}+\frac{D_2}{\lambda_2-\alpha M_2}\right],
\end{split}\]
for $t\in \R$. Let $C>0$ be given. Then, if $\varepsilon$ is such that~\eqref{cor2cond1} is satisfied, it follows from Theorem~\ref{t3} that~\eqref{Con} holds.

We now prove~(b).   Let $t\mapsto (x(t),y(t))$ and
$t\mapsto (z(t),w(t))$ be solutions of (\ref{Heq2}) such that $x(s)=z(s)$. 
Then in Appendix~\ref{A3} it is proved that
\[ |x(t)-z(t)|+|y(t)-w(t)|\le 2e^{(M_3+K_1\varepsilon)|t-s|}|y(s)-w(s)|\]
if $t\ge s$ and
\[ |x(t)-z(t)|+|y(t)-w(t)|\le 2e^{(M_3+K_2\varepsilon)|t-s|}|y(s)-w(s)|\]
if $t\le s$. Hence, \eqref{Del} holds with
\[ \Delta_3(t,s)=\begin{cases} 2e^{(M_3+K_1\varepsilon)|t-s|} & (t\ge s)\\
                             2e^{(M_3+K_2\varepsilon)|t-s|} & (t\le s).\end{cases}\]
Then, 
\[\begin{split}
&\int^{\infty}_{-\infty}|\mathcal G(t,s)|\Delta^{\alpha}_3(s,t)\, ds\\ 
&\le 2^\alpha \displaystyle \int^{t}_{-\infty}
D_1e^{-\lambda_1(t-s)}e^{\alpha(M_3+K_2\varepsilon)(t-s)}\, ds
+2^\alpha \int^{\infty}_{t}D_2e^{-\lambda_2(s-t)}e^{\alpha(M_3+K_1\varepsilon)(s-t)}\, ds\\ 
&\le \displaystyle  2^\alpha\left[\frac{D_1}{\lambda_1-\alpha(M_3+K_2\varepsilon)}+
\frac{D_2}{\lambda_2-\alpha (M_3+K_1\varepsilon)}\right].
\end{split}\]
Given $C>0$, if $\varepsilon$ is satisfies~\eqref{cor2condition3}, it follows from Theorem~\ref{t3} that~\eqref{Con2} holds. The proof of the corollary is completed.
\end{proof}

\begin{remark}\label{Obs}
Under the assumptions of the previous two corollaries we have (see Remark~\ref{rm3}) 
that $H$ and $\bar H$ are H\" older continuous in $(x,y)$ on each bounded subset of 
$X\times Y$. A similar result is established in~\cite[Theorem 2]{SX} 
(see also~\cite[Theorem 2.2.]{XWKR}). In contrast to these results, we 
obtain H\" older continuity on each bounded subset of $X\times Y$, and not only 
on the unit ball on $X\times Y$. Finally, our proofs of 
Theorems~\ref{t2} and~\ref{t3} are somewhat simpler than the corresponding arguments 
in~\cite{SX, XWKR}.
\end{remark}

\begin{remark}
In addition, we also provide an upper bound for the 
H\" older exponent in $\alpha$ in terms of the ratios between the dichotomy
exponents and the growth and decay rates for $\dot x=A(t)x$, in the case of H\" older
continuity in $x$. In the case of H\" older continuity in $y$, the growth and decay rates 
of solutions of $\dot y=g(t,y)$ must also be taken into account.
\medskip

For $\dot x=A(t)x$, the dichotomy exponents and the growth and decay rates can
all be seen from the Sacker-Sell spectrum $\Sigma$, which is the set of real $\lambda$ for
which $\dot x=(A(t)-\lambda I)x$ does not have an exponential dichotomy. 
$\Sigma$ is a closed set (see \cite{S}). We suppose in addition it is bounded. Since we are assuming
$\dot x=A(t)x$ has an exponential dichotomy, then $0\notin \Sigma$. In our corollaries
above, we have the four positive numbers $\lambda_1\le a_2$, $\lambda_2\le a_1$.
It turns out we may take them as any numbers satisfying
\[ -a_2< \inf\Sigma \le \sup[\Sigma\cap(-\infty,0)]<-\lambda_1<0<\lambda_2
   < \inf[\Sigma\cap(0,\infty)]\le \sup\Sigma<a_1.\]

\end{remark}

\subsection{Examples without exponential dichotomy}
We now discuss the applicability of Theorems~\ref{t2} and~\ref{t3} in the case when $\dot x=A(t)x$ does not admit an exponential dichotomy.

\begin{example}
We recall the notion of an exponential trichotomy introduced by Elaydi and Hajek~\cite{EH}. Assume that $X=\R^d$. We say that $\dot x=A(t)x$ admits an exponential trichotomy if there exist two families of projections 
$(P^+(t))_{t\ge 0}$ and $(P^-(t))_{t\le 0}$ and positive constants $D_i, \lambda_i$, $i\in \{1, 2, 3, 4\}$ such that the following conditions hold:
\begin{itemize}
\item $P^+(t)T(t,s)=T(t,s)P^+(s)$ for $t, s\ge 0$ and $P^-(t)T(t,s)=T(t,s)P^-(s)$ for $t, s\le 0$;
\item for $t\ge s\ge 0$, 
\[
|T(t,s)P^+(s)|\le D_1e^{-\lambda_1(t-s)} \quad \text{and} \quad |T(s,t)(I-P^+(t))|\le D_2e^{-\lambda_2(t-s)};
\]
\item for $0\ge t\ge s$, 
\[
|T(t,s)P^-(s)|\le D_3e^{-\lambda_3(t-s)} \quad \text{and} \quad |T(s,t)(I-P^-(t))|\le D_4e^{-\lambda_4(t-s)};
\]
\item $P^-(0)=P^-(0)P^+(0)=P^+(0)P^-(0)$.
\end{itemize}
Obviously, if $\dot x=A(t)x$ admits an exponential dichotomy, then it also admits an exponential trichotomy. However, the converse statement does not  hold. 

Assume that $\dot x=A(t)x$ admits an exponential trichotomy and that the first two estimates in~\eqref{est-1} hold (with some $K_i, a_i>0$, $i=1,2$). Furthermore, let $\mathcal G(t,s)$ be as in~\eqref{G} with 
\[
P(s)=\begin{cases}
P^+(s) & s\ge 0; \\
P^-(s) & s<0.
\end{cases}
\]
It is easy to show that the first estimate in~\eqref{r} holds. 
Similarly to the arguments in the previous subsection, one can formulate conditions under which Theorem~\ref{t2} and~\ref{t3} can be applied in the present setting.

Finally, we note the following: if $X$ is finite-dimensional and the first two estimates in~\eqref{est-1} hold (with some $K_i, a_i>0$, $i=1,2$), then $\dot x=A(t)x$ admits an exponential trichotomy if and only if the first estimate in~\eqref{r} holds. Indeed, this follows from~\cite[Theorem 5.2]{EH}.
\end{example}

 We now consider an example studied (for a different purpose) by Coppel~\cite[p.27]{C}. 
\begin{example}
Let $X=\R$. We choose a continuously differentiable function $\phi \colon \R \to \R$ with the following properties:
\begin{itemize}
\item $\phi(t) \in (0, 1]$ for each $t\in \R$ and $\phi(t)=1$ for $t\le 0$;
\item $\int_0^\infty (1/\phi(s)-1) \, ds<+\infty$;
\item \begin{equation}\label{lim} \lim_{n\to \infty}\frac{\phi(n)}{\phi(n-2^{-n})}=+\infty. \end{equation}
\end{itemize}
Furthermore, for $t\in \R$ set 
\[
A(t)=\frac{\phi'(t)}{\phi(t)}-1 \quad \text{and} \quad P(t)=I.
\]
Then, 
\[
T(t,s)= \frac{\phi(t)}{\phi(s)} e^{-(t-s)} \quad \text{and} \quad 
\mathcal G(t,s)=\begin{cases}
\frac{\phi(t)}{\phi(s)} e^{-(t-s)} & t\ge s; \\
0 & t<s.
\end{cases}
\]
It follows easily from~\eqref{lim} that $\dot x=A(t)x$ does not admit an exponential dichotomy (since it is not uniformly stable).

We now assume that there exist $M\ge 1$ and $\varepsilon >0$ such that~\eqref{a1} and~\eqref{EPs} hold. Observe that for every $t\ge 0$, 
\begin{equation}\label{leb}
\begin{split}
\int_0^t \frac{\phi(t)}{\phi(s)} e^{-(t-s)} \, ds &=e^{-t}\phi(t) \int_0^t \frac{e^s}{\phi(s)}\, ds  \\
&=e^{-t}\phi(t) \int_0^t e^s\, ds+e^{-t} \phi(t) \int_0^t e^s (1/\phi(s)-1)\, ds  \\
&\le 1+\int_0^\infty (1/\phi(s)-1)\, ds,
\end{split}
\end{equation}
and thus 
\[
\sup_{t\in \R} \int_{-\infty}^\infty |\mathcal G(t,s) | \, ds \le 2+\int_0^\infty (1/\phi(s)-1)\, ds <+\infty.
\]
Assuming that $\varepsilon$ satisfies 
\[
\varepsilon \bigg (2+\int_0^\infty (1/\phi(s)-1)\, ds \bigg )<1, 
\]
we have that~\eqref{bound} holds with $\mu(t)=1$ and $\gamma(t)=\varepsilon$.  Consequently, there exist $H$ and $\bar H$ as in the statement of Theorem~\ref{t1}. 

Take now $\alpha \in (0, 1)$ and $C>0$. By arguing as in~\eqref{leb}, one can conclude that 
\[
\sup_{t\in \R} \int_{-\infty}^\infty \mathcal G(t,s)\Delta_1^\alpha (s, t)\, ds=\sup_{t\in \R}\int_{-\infty}^\infty \mathcal G(t,s)T(s, t)^\alpha \, ds \le \frac{2}{1-\alpha} +\int_0^\infty (1/\phi(s)-1)\, ds. 
\]
Therefore, provided that $\varepsilon >0$ satisfies 
\[
\max \{2M, \varepsilon \} (1+C)\varepsilon^\alpha  \bigg ( \frac{2}{1-\alpha} +\int_0^\infty (1/\phi(s)-1)\, ds \bigg ) \le C,
\]
it follows from Theorem~\ref{t2} that~\eqref{conc1} holds. 

On the other hand, by arguing  as in Appendix~\ref{A2}, one can easily show that 
\[
\Delta_2 (t,s)=\frac{\phi(t)}{\phi(s)} e^{-(1+\varepsilon)(t-s)}, \quad t\le s.
\]
Then, provided that $\varepsilon$ is sufficiently small, we have that 
\[
\sup_{t\in \R} \int_{-\infty}^\infty \mathcal G(t,s)\Delta_2^\alpha (s,t) \, ds \le \frac{2}{1-\alpha (1+\varepsilon)}+\int_0^\infty (1/\phi(s)-1)\, ds.
\]
Hence, if $\varepsilon>0$ satisfies 
\[
2M\varepsilon^\alpha \bigg (\frac{2}{1-\alpha (1+\varepsilon)}+\int_0^\infty (1/\phi(s)-1)\, ds \bigg ) \le C,
\]
it follows from Theorem~\ref{t2} that~\eqref{conc2} holds.  Similarly,  one can discuss the applicability of Theorem~\ref{t3} in this setting. 
\end{example}

\begin{remark}
The previous two examples show that Theorem~\ref{t2} and~\ref{t3} can be applied in situations when $\dot x=A(t)x$ does not admit an exponential dichotomy. This shows that our results are more general than those in~\cite{BV, SX, XWKR}.
\end{remark}

\section{The case of discrete time}\label{DT}
Let $X$ and $Y$ be as in Subsection~\ref{Pre}. Let $(A_n)_{n\in \Z}$ be a sequence of invertible operators in $\mathcal B(X)$ and  $f_n \colon X\times Y \to X$, $n\in \Z$ be a sequence of maps with the property that there exist sequences $(\mu_n)_n$ and $(\gamma_n)_n$ in $[0, \infty)$ such that 
\begin{equation}\label{dis1}
|f_n(x, y)|\le \mu_n \quad \text{and} \quad |f_n(x_1, y)-f_n(x_2, y)|\le \gamma_n |x_1-x_2|,
\end{equation}
for $n\in \Z$, $x, x_1, x_2\in X$ and $y\in Y$. Furthermore, let $(g_n)_{n\in \Z}$ be a sequence of homeomorphisms on $Y$. We consider a coupled system 
\begin{equation}\label{nnd}
x_{n+1} =A_nx_n+f_n(x_n, y_n), \quad y_{n+1}=g_n( y_n).
\end{equation}
For $m, n\in \Z$, set 
\[
\mathcal A(m, n)=\begin{cases}
A_{m-1}\cdots A_n & \text{for $m>n$;}\\
I  &\text{for $m=n$;} \\
A_m^{-1}\cdots A_{n-1}^{-1} & \text{for $m<n$.}
\end{cases}
\]
Let $(P_n)_{n\in \Z}$ be a sequence in $\mathcal B(X)$. We define 
\begin{equation}\label{G-d}
\mathcal G(m, n)=\begin{cases}
\mathcal A(m, n)P_n & \text{for $m\ge n$;}\\
-\mathcal A(m, n)(I-P_n) & \text{for $m<n$.}
\end{cases}
\end{equation}
\subsection{A linearization result}
Besides~\eqref{nnd}, we also consider the uncoupled system
\begin{equation}\label{lnd}
x_{n+1}=A_n x_n, \quad y_{n+1}=g_n(y_n).
\end{equation}
Let  $k\mapsto (x_1(k, n, x), y(k, n, y))$ denote the solution of~\eqref{lnd} which equals $(x, y)$ when $k=n$. Similarly, $k\mapsto (x_2(k, n, x, y), y(k, n, y))$ is the solution of~\eqref{nnd}
whose value is $(x, y)$ at $k=n$.

The following result is a version of Theorem~\ref{t1} in the present setting.
\begin{theorem}\label{t1d}
Suppose that  
\begin{equation}\label{boundd}
N:=\sup_{m\in \Z} \sum^{\infty}_{n=-\infty}|\mathcal G(m,n)|\mu_{n-1} <\infty  \quad \text{and} \quad q:= \sup_{m\in \Z}\sum^{\infty}_{n=-\infty}|\mathcal G(m,n)|\gamma_{n-1} <1.
\end{equation}
Then, there exists a sequence of continuous functions $H_n\colon X\times Y\to X\times Y$, $n\in \Z$ of the form $H_n(x, y)=(x+h_n(x, y), y)$, where  $\sup_{n,x,y}|h_n(x, y)|<\infty$, such that if $n\mapsto (x_n,y_n)$ is a solution of (\ref{lnd}),
then $n\mapsto H_n(x_n, y_n)$ is a solution of (\ref{nnd}). In addition, there exists a sequence of continuous functions $\bar H_n \colon X\times Y\to X\times Y$, $n\in \Z$ of the form $\bar H_n( x, y)=(x+\bar h_n(x, y), y)$, where $\sup_{n,x, y}|\bar h_n(x, y)|<\infty$,  such that if $n\mapsto (x_n, y_n)$ is a solution of (\ref{nnd}),
then $n\mapsto \bar H_n(x_n, y_n)$ is a solution of (\ref{lnd}). Moreover, $H_n$ and $\bar H_n$ are inverses of
each other, that is,
\begin{equation}\label{inv} H_n(\bar H_n(x,y))=(x,y)=\bar H_n(H_n(x,y)),\end{equation}
for $n\in \Z$ and $(x, y)\in X\times Y$.
\end{theorem}

\begin{proof}
Since the proof is analogous to that of Theorem~\ref{t1}, we will only provide a sketch of the argument.
Let $\mathcal Z$ be the space of all sequences $\mathbf h=(h_n)_{n\in \Z}$ of continuous maps $h_n\colon X\times Y\to X$, $n\in \Z$ such that
\[
\lVert \mathbf h\rVert_{\infty}:=\sup_{n, x,y}|h_n(x,y)| <\infty. 
\]
Then, $(\mathcal Z, \lVert \cdot \rVert_{\infty})$ is a Banach space.  For $\mathbf h=(h_n)_{n\in \Z}\in \mathcal Z$, we define $\hat{\mathbf h}=(\hat h_n)_{n\in \Z}$ by 
\[
\hat h_n(\xi, \eta)=\sum_{k=-\infty}^\infty \mathcal G(n, k) H(k, n)(\xi, \eta) \quad \text{for $n\in \Z$ and $(\xi, \eta)\in X\times Y$,}
\]
where
\begin{equation}\label{ff}
\begin{split}
&H(k, n)(\xi, \eta) \\
&=f_{k-1} (x_1(k-1, n, \xi)+h_{k-1}(x_1(k-1, n, \xi), y(k-1, n, \eta)), y(k-1, n, \eta)).
\end{split}
\end{equation}
Observe that it follows from~\eqref{dis1} and~\eqref{boundd} that 
\[
|\hat h_n(\xi, \eta) |\le \sum_{k=-\infty}^\infty |\mathcal G(n,k)|\mu_{k-1}\le N,
\]
for $n\in \Z$ and $(\xi, \eta)\in X\times Y$. Thus, $\hat{\mathbf h}\in \mathcal Z$. In addition, \eqref{dis1} and~\eqref{boundd} imply that 
\[
\lVert \hat{\mathbf h}^1-\hat{\mathbf h}^2\rVert_\infty \le q\lVert \mathbf h^1-\mathbf h^2\rVert_\infty, \quad \text{for $\mathbf h^i\in \mathcal Z$, $i=1, 2$.}
\]
We conclude that the map $T\colon \mathcal Z\to \mathcal Z$ given by $T(\mathbf h)=\hat{\mathbf h}$, $\mathbf h\in \mathcal Z$ is a contraction. Consequently, $T$ has a unique fixed point $\mathbf h=(h_n)_{n\in \Z}\in \mathcal Z$.
Then, 
\[
h_n (\xi, \eta)=\sum_{k=-\infty}^\infty \mathcal G(n, k) H(k, n)(\xi, \eta), 
\]
for $n\in \Z$ and $(\xi, \eta)\in X\times Y$, 
where $H(k, n)(\xi, \eta)$ is given by~\eqref{ff}. Then, we have that 
\[
h_n (x_1(n, m, \xi), y(n, m, \eta))=\sum_{k=-\infty}^\infty \mathcal G(n, k) H(k, n)
(x_1(n, m, \xi), y(n, m, \eta)).
\]
Moreover, since 
\begin{equation}\label{ident}x_1(k-1,n,x_1(n,m,\xi))=x_1(k-1,m,\xi),\;
y(k-1,n,y(n,m,\eta))=y(k-1,m,\eta),\end{equation}
we have that 
\[
\begin{split}
&H(k, n)(x_1(n, m, \xi), y(n, m, \eta))\\
&=f_{k-1}(x_1(k-1, m, \xi)+h_{k-1}(x_1(k-1, m, \xi), y(k-1, m, \eta)), y(k-1, m, \eta)).
\end{split}
\]
It follows that if $n\mapsto (x_n, y_n)$ is  a solution of~\eqref{lnd}, then 
\[
h_n(x_n, y_n)=\sum_{k=-\infty}^\infty \mathcal G(n, k) f_{k-1}(x_{k-1}+h_{k-1}
(x_{k-1},y_{k-1}), y_{k-1}), \ n\in \Z
\]
so that
\[\begin{array}{rl}
h_n(x_n, y_n)
&=\displaystyle\sum_{k=-\infty}^n{\mathcal A}(n, k)P_kf_{k-1}(x_{k-1}+h_{k-1}
(x_{k-1},y_{k-1}), y_{k-1})\\ \\
&\qquad-\displaystyle\sum^{k=\infty}_{n+1}{\mathcal A}(n, k)(I-P_k)f_{k-1}(x_{k-1}+h_{k-1}
(x_{k-1},y_{k-1}), y_{k-1}).\end{array}
\]
Consequently, 
\[
\begin{split}
& x_{n+1}+h_{n+1}(x_{n+1}, y_{n+1})-A_n(x_n+h_n (x_n, y_n)) \\
&=h_{n+1}(x_{n+1}, y_{n+1})-A_n h_n (x_n, y_n) \\
&=\displaystyle A_n\left[\sum_{k=-\infty}^{n+1}{\mathcal A}(n, k)P_kf_{k-1}(x_{k-1}+h_{k-1}
(x_{k-1},y_{k-1}), y_{k-1})\right.\\
&\qquad-\left.\displaystyle\sum^{k=\infty}_{n+2}{\mathcal A}(n, k)(I-P_k)f_{k-1}(x_{k-1}+h_{k-1}
(x_{k-1},y_{k-1}), y_{k-1})\right]\\
&\qquad\qquad-A_nh_n(x_n,y_n)\\
&=A_nh_n(x_n,y_n)+P_{n+1}f_{n}(x_{n}+h_{n}(x_{n},y_{n}), y_{n})
+(I-P_{n+1})f_{n}(x_{n}+h_{n}(x_{n},y_{n}), y_{n})\\
&\qquad -A_nh_n(x_n,y_n)\\
&=f_{n}(x_{n}+h_{n}(x_{n},y_{n}), y_{n}),
\end{split}
\]
for $n\in \Z$. Therefore, $n\mapsto (x_n+h_n(x_n, y_n),y_n)$ is a solution of~\eqref{nnd}. For $n \in \Z$ and $(\xi, \eta)\in X\times Y$, we set
\[ H_n(\xi,\eta)=(\xi+h_n(\xi,\eta),\eta).\]
From the preceding discussion, we have that if $n\mapsto (x_n,y_n)$ is a solution of (\ref{lnd}), then
$n\mapsto H_n(x_n,y_n)$ is a solution of (\ref{nnd}).

Now set
\[
\bar h_n(\xi, \eta)=-\sum_{k\in \Z} \mathcal G(n, k)f_{k-1}(x_2(k-1, n, \xi, \eta), y(k-1, n, \eta)),
\]
for $n\in \Z$ and $(\xi, \eta)\in X\times Y$. Using~\eqref{boundd}, 
one can easily see that $(\bar h_n)_{n\in \Z}\in \mathcal Z$.  
Observe, using a similar identity to (\ref{ident}), that
  \[\bar h_n(x_2(n, m, \xi,\eta),y(m,n,\eta))
=-\sum^{\infty}_{k=-\infty}\mathcal G(n,k)f_{k-1}(x_2(k-1,m,\xi,\eta),y(k-1,m,\eta)).\]
Hence, if $n\mapsto (x_n,y_n)$ is a solution of (\ref{nnd}), we have that 
\begin{equation}\bar h_n(x_n, y_n)
= -\sum^{\infty}_{k=-\infty}\mathcal G(n,k)f_{k-1}(x_{k-1}, y_{k-1}), \quad m\in \Z.\end{equation}
By direct calculation,
\[ \bar h_{n+1}(x_{n+1},y_{n+1})=A_n\bar h_n(x_n,y_n)-f_n(x_n,y_n) \quad n\in \Z,\]
and thus $n\mapsto (x_n+\bar h_n(x_n,y_n),y_n)$ is a solution of (\ref{lnd}). For $n \in \Z$ and $(\xi, \eta)\in X\times Y$, set
\[ \bar H_n(\xi,\eta)=(\xi+\bar h_n(\xi,\eta),\eta).\]
By the  preceding discussion,  if $n\mapsto (x_n,y_n)$ is a solution of (\ref{nnd}), then 
$n\mapsto \bar H_n(x_n,y_n)$ is a solution of (\ref{lnd}). Finally, by arguing as in the proof of Theorem~\ref{t1} one can show that~\eqref{inv} holds.
\end{proof}
The following example is a discretization of Example~\ref{exp}. 
\begin{example}
Let $X=\R^3$ and for $n\in \Z$, set 
\[
A_n=\begin{pmatrix}
\cos 1 & -\sin  1  & 0 \\
\sin 1 & \cos 1 & 0 \\
0 & 0 & \frac{1+n^2}{1+(n+1)^2}
\end{pmatrix} \quad \text{and} \quad
P_n=\begin{pmatrix}
0 & 0 & 0 \\
0 & 0 & 0 \\
0 & 0 & 1
\end{pmatrix}.
\]
Then, 
\[
\mathcal G(m, n)=\begin{cases}
\begin{pmatrix} 
0 & 0 & 0 \\
0 & 0 &0 \\
0 &0 &\frac{1+n^2}{1+m^2}
\end{pmatrix}  & \text{for $m\ge n$;} \\
-\begin{pmatrix}
\cos (m-n) & -\sin (m-n) & 0 \\
\sin (m-n) &\cos (m-n) & 0 \\
0 &0 &0
\end{pmatrix} &\text{for $m<n$.}
\end{cases}
\]
Hence, $|\mathcal G(m,n) |\le 1+n^2$ and thus~\eqref{boundd} holds whenever
\[
 \sum^{\infty}_{n=-\infty}(1+n^2)\mu_{n-1} <\infty  \quad \text{and} \quad \sum^{\infty}_{n=-\infty}(1+n^2)\gamma_{n-1} <1.
\]
Finally, we observe that the difference equation $x_{n+1}=A_n x_n$ does not  admit an exponential dichotomy (or even an ordinary  dichotomy).
\end{example}

\subsection{H\" older continuity  in $x$}
We suppose there exist $\Delta_1(m,n)>0$ such that
\begin{equation}\label{Td} |\mathcal A(m, n)|\le \Delta_1(m,n)\end{equation}
and $\Delta_2(m,n)>0$ such that
if $n\mapsto (x_n, y_n)$ and $n\mapsto (z_n,y_n)$ are solutions of (\ref{nnd})
then 
\begin{equation}\label{ood}  |x_m-z_m|\le \Delta_2(m,n)|x_n-z_n|.\end{equation}
Furthermore, we assume there exists $M\ge 1$ and a sequence $(\varepsilon_n)_{n\in \Z}\subset [0, \infty)$  such that 
\begin{equation}\label{a1d} |f_n(x,y)|\le M\end{equation}
and
\begin{equation}\label{a2d} |f_n(x_1,y)-f_n(x_2,y)| \le \varepsilon_n |x_1-x_2|,\end{equation}
for $n\in \Z$, $x, x_1, x_2\in X$ and $y\in Y$. In addition, we assume that there exists $N\ge 1$  such that \begin{equation}\label{apd} \varepsilon_n \le N.\end{equation} Finally, we suppose that
\begin{equation}\label{rd} \sup_{m\in \Z} \sum^{\infty}_{n=-\infty}|\mathcal G(m,n)|<\infty, \quad 
q:=\sup_{m\in \Z}  \sum^{\infty}_{n=-\infty}|\mathcal G(m,n)|\varepsilon_{n-1} <1.\end{equation}
Observe that the above condition implies that~\eqref{boundd} holds and thus Theorem~\ref{t1d} is valid. Hence,  there is a sequence $H_n(x,y)=(x+h_n(x,y),y)$ sending the solutions of
(\ref{lnd}) onto the solutions of (\ref{nnd}) and a sequence  $\bar H_n(x,y)=(x+\bar h_n(x,y),y)$ sending the solutions of
(\ref{nnd}) onto the solutions of (\ref{lnd}).

\begin{theorem}\label{t2d}
 Let $C>0$ and $0<\alpha<1$ be given. Then if 
\begin{equation}\label{c1d} \max\{2M,N\}(1+C)\sup_{n\in \Z}\sum^{\infty}_{m=-\infty}|\mathcal G(n,m)|\varepsilon_{m-1}^{\alpha}\Delta^{\alpha}_1(m-1,n)
\le C,\end{equation}
we have
\[ |h_n(x_1,y)-h_n(x_2,y)|\le C|x_1-x_2|^{\alpha}, \quad \text{for $n\in \Z$, $x_1, x_2\in X$ and $y\in Y$.}\]
Moreover, if
\begin{equation}\label{c2d} 2M \sup_{n\in \Z} \sum^{\infty}_{m=-\infty}|\mathcal G(n,m)|\varepsilon_{m-1}^{\alpha}\Delta^{\alpha}_2(m-1,n)
\le C,\end{equation}
then
\[ |\bar h_n(x_1,y)-\bar h_n(x_2,y)|\le C|x_1-x_2|^{\alpha}, \quad \text{for $n\in \Z$, $x_1, x_2\in X$ and $y\in Y$.}\]
\end{theorem}

\begin{proof}
We follow closely the proof of Theorem~\ref{t2}. We begin by observing that~\eqref{a1d} and~\eqref{a2d} imply that 
\begin{equation}\label{Heq9d} \begin{array}{rl}
|f_n(x,y)-f_n(z,y)|
&=|f_n(x,y)-f_n(z,y)|^{1-\alpha}|f_n(x,y)-f_n(z,y)|^{\alpha}\\ \\
&\le 2M\varepsilon_n^{\alpha}|x-z|^{\alpha},
\end{array}\end{equation}
for $n\in \Z$, $x, z\in X$ and $y\in Y$. 

Let $\mathcal Z$ be as in the proof of Theorem~\ref{t1d}. Furthermore, let $\mathcal Z'$ denote the set of all $\Psi=(\psi_n)_{n\in \Z} \in \mathcal Z$ such that
\[
|\psi_n(x_1, y)-\psi_n(x_2, y)|\le C|x_1-x_2|^\alpha, \quad \text{for $n\in \Z$, $x_1, x_2\in X$ and $y\in Y$.}
\]
Then, $\mathcal Z'$ is a closed subset of $\mathcal Z$. We now prove that 
$T(\mathcal Z')\subset \mathcal Z'$, where $T$ is as in the proof of Theorem~\ref{t1d}.  We recall that 
\[
(T \Psi)_n(\xi, \eta)=\sum_{k=-\infty}^\infty \mathcal G(n,k)H(k, n)(\xi, \eta) \quad \text{for $n\in \Z$ and $(\xi, \eta)\in X\times Y$,}
\]
where
\[
\begin{split}
&H(k, n)(\xi, \eta) \\
&=f_{k-1} (x_1(k-1, n, \xi)+\psi_{k-1}(x_1(k-1, n, \xi), y(k-1, n, \eta)), y(k-1, n, \eta)).
\end{split}
\]
Take an arbitrary $\Psi=(\psi_n)_{n\in \Z}\in \mathcal Z'$. 
We observe  (using~\eqref{a2d}, \eqref{apd} and~\eqref{Heq9d})  that 
\begin{equation}\label{tatd}
|f_n(x_1+\psi_n(x_1,y),y)-f_n(x_2+\psi_n(x_2,y),y) | 
\le M_1\varepsilon_n^{\alpha}(1+C)|x_1-x_2|^{\alpha},\end{equation}
where $M_1=\max\{N,2M\}$, following the argument in (\ref{tat}) with $t$ replaced by $n$. Then, 
\[ (T\Psi)_n(\xi_1,\eta)-(T\Psi)_n(\xi_2,\eta)=\sum^{\infty}_{m=-\infty}\mathcal G(n,m)p(m),\]
where
\[\begin{split} p(m)
&= f_{m-1} (x_1(m-1, n, \xi_1)+\psi_{m-1}(x_1(m-1, n, \xi_1), y(m-1, n, \eta)), y(m-1, n, \eta))\\ 
&\phantom{=} -f_{m-1} (x_1(m-1, n, \xi_2)+\psi_{m-1}(x_1(m-1, n, \xi_2), y(m-1, n, \eta)), y(m-1, n, \eta)).
\end{split}\]
Using~\eqref{tatd}, we see that 
\[ \begin{split}
|p(m)|
&\le M_1\varepsilon_{m-1}^{\alpha}(1+C)|x_1(m-1,n,\xi_2)-x_1(m-1,n,\xi_1)|^{\alpha}\ \\ 
&\le M_1\varepsilon_{m-1}^{\alpha}(1+C)[\Delta_1(m-1,n)|\xi_1-\xi_2|]^{\alpha}\\ 
&= M_1\varepsilon_{m-1}^{\alpha}(1+C)\Delta^{\alpha}_1(m-1,n)|\xi_1-\xi_2|^{\alpha}.
\end{split}\]
Therefore, by~\eqref{c1d} we have that 
\[\begin{split}
& | (T\Psi)_n(\xi_1,\eta)-(T\Psi)_n(\xi_2,\eta)|\\ 
&\le \displaystyle M_1(1+C)|\xi_1-\xi_2|^{\alpha} \sup_{n\in \Z}\sum^{\infty}_{m=-\infty}|\mathcal G(n,m)|\varepsilon_{m-1}^{\alpha}\Delta^{\alpha}_1(m-1, n)
\\ 
&\le C|\xi_1-\xi_2|^{\alpha},
\end{split}\]
for $t\in \R$, $\xi_1, \xi_2\in X$ and $\eta \in Y$. Therefore, $T\Psi \in \mathcal Z'$. Consequently, the unique fixed point $(h_n)_{n\in \Z}$ of $T$ belongs to $\mathcal Z'$, which implies the first assertion of the theorem.

On the other hand, we recall from Theorem~\ref{t1d} that 
\[\bar h_n(\xi,\eta)=-\sum_{m\in \Z} \mathcal G(n, m)f_{m-1}(x_2(m-1, n, \xi, \eta), y(m-1, n, \eta)).
\]
Then
\[ \bar h_n(\xi_1,\eta)-\bar h_n(\xi_2,\eta)
=\sum^{\infty}_{m=-\infty}\mathcal G(n,m)p(m),\]
where
\[ \begin{split}
p(m) &=f_{m-1}(x_2(m-1,n,\xi_2,\eta),y(m-1,n,\eta)) \\
&\phantom{=}-f_{m-1}(x_2(m-1,n,\xi_1,\eta),y(m-1,n,\eta)). \end{split}\]
By~\eqref{Heq9d}, we have that 
\[ \begin{split} |p(m)| &\le 2M\varepsilon_{m-1}^{\alpha}|x_2(m-1,n,\xi_2,\eta)-x_2(m-1,n,\xi_1,\eta)|^{\alpha} \\
&\le 2M\varepsilon_{m-1}^{\alpha}\Delta_2^{\alpha}(m-1,n)|\xi_1-\xi_2|^{\alpha}. \end{split}\]
Consequently, using~\eqref{c2d} we conclude that 
\[\begin{split}
& |\bar h_n(\xi_1,\eta)-\bar h_n(\xi_2,\eta)|\\ 
&\le \displaystyle 2M |\xi_1-\xi_2|^{\alpha}\sum^{\infty}_{m=-\infty}|\mathcal G(n,m)|\varepsilon_{m-1}^{\alpha}\Delta^{\alpha}_2(m-1,n)
\\ 
&\le C|\xi_1-\xi_2|^{\alpha},
\end{split}\]
for $n\in \Z$, $\xi_1, \xi_2\in X$ and $\eta\in Y$, which establishes the second assertion of the theorem.  The proof of the theorem is completed.
\end{proof}

\subsection{H\" older continuity  in $y$}
We continue to assume that~\eqref{a1d} holds with $M\ge 1$.  Moreover, we assume that there exists a sequence $(\varepsilon_n)_{n\in \Z} \subset [0, +\infty)$ satisfying~\eqref{apd} (with some $N\ge 1$) and such that 
\[|f_n(x_1,y_1)-f_n(x_1,y_2)|\le \varepsilon_n [|x_1-x_2|+|y_1-y_2|],\]
for $n\in \Z$, $x_1, x_2\in X$ and $y_1, y_2\in Y$.
 Furthermore, we suppose that if  $n\mapsto y_n$ and $n\mapsto w_n$ are solutions of $y_{n+1}=g_n(y_n)$, then 
\[ |y_m-w_m|\le \sigma(m,n)|y_n-w_n|,\]
for some $\sigma(m,n)>0$. In addition, we assume that there exist $\Delta_3(m,n)>0$  such that if $m\mapsto (x_m, y_m)$
and $m\mapsto (z_m, w_m)$ are solutions of (\ref{nnd}) with $x_n=z_n$, then 
\[ |x_m-z_m|+|y_m-w_m|\le \Delta_3(m,n)|y_n-w_n|.\]
Finally, suppose that~\eqref{rd} holds. Hence, Theorem~\ref{t1} is again applicable. The proof of the following result is similar to the proofs of Theorems~\ref{t3} and~\ref{t2d} and therefore it is omitted. 
\begin{theorem} \label{t3d}
 Let $C>0$ and $0<\alpha<1$ be given. Then, if 
\[
\max\{2M,N\}(1+C)\sup_{n\in \Z}\sum^{\infty}_{m=-\infty}|\mathcal G(n,m)|\varepsilon_{m-1}^{\alpha}\sigma^{\alpha}(m-1,n)
	\le C,
\]
we have that
	\[ |h_n(x,y_1)-h_n(x,y_2)|\le C|y_1-y_2|^{\alpha}, \quad \text{for $n\in \Z$, $x\in X$ and $y_1, y_2\in Y$.}\]
Moreover, if \[\label{c4} 2M\sup_{n\in \Z} \sum^{\infty}_{m=-\infty}|
\mathcal G(n,m)|\varepsilon_{m-1}^{\alpha}\Delta^{\alpha}_3(m-1,n)
	\le C,\]
then 
\[ |\bar h_n(x,y_1)-\bar h_n(x,y_2)|\le C|y_1-y_2|^{\alpha}, \quad \text{for $n\in \Z$, $x\in X$ and $y_1, y_2\in Y$.}\]
\end{theorem}

\begin{remark}
Like in Remark \ref{rm1} (Remark \ref{rm2}, resp.) we can get that, under the assumptions of Theorem \ref{t2d} (Theorem \ref{t3d}, resp.), for each $n\in \Z$, both $H_n(x,y)$ and $\bar{H}_n(x,y)$ are H\"older continuous with respect to $x$ ($y$, resp.) when restricted to any bounded subset of $X$ ($Y$, resp.). Moreover, as in Remark \ref{rm3}, under the assumptions of Theorems \ref{t2d} and \ref{t3d}, we may conclude that for each $n\in \Z$, $H_n(x,y)$ and $\bar{H}_n(x,y)$ are H\"older continuous in $(x,y)$ when restricted to any bounded subset of $X\times Y$. 
\end{remark}

\begin{remark}\label{EDD}
As in Section \ref{EX}, we can easily apply Theorems \ref{t1d}, \ref{t2d} and 
\ref{t3d} to the context where the sequence of operators $(A_n)_{n\in \Z}$ admits an 
exponential dichotomy, that is, there exists a sequence $(P_n)_n$ of projections on 
$X=\R^n$ 
satisfying 
\begin{equation*}
{\mathcal A}(n,m)P_m=P_{n}{\mathcal A}(n,m) \quad \text{for $n, m\in \Z$,}
\end{equation*}
and positive constants $D_1$, $D_2$, $\lambda_1$, $\lambda_2$ such that for $n\ge m$
\[ |{\mathcal A}(n,m)P_m|\le D_1e^{-\lambda_1(n-m)},\quad
   |{\mathcal A}(m,n)(I-P_n)|\le D_2e^{-\lambda_2(n-m)},\]
and bounded growth and decay conditions
\[ |{\mathcal A}(n,m)|\le K_1e^{a_1(n-m)},\quad
    |{\mathcal A}(m,n)| \le K_2e^{a_2(n-m)},\]
where $K_1$, $K_2$, $a_1$, $a_2$ are positive constants such that $a_2\ge \lambda_1$,
$a_1\ge \lambda_2$. We skip the details. 
\end{remark}

\begin{remark}
We note that  Theorems~\ref{t1d} and~\ref{t2d} can also be applied in the case when $(A_n)_{n\in \Z}$ admits the so-called generalized exponential dichotomy~\cite{BD}. We recall that a sequence $(A_n)_{n\in \Z} \subset \mathcal B(X)$ of 
invertible operators is said to admit a generalized exponential dichotomy if 
\begin{itemize}
\item for each $n\in \Z$ there are closed subspaces $S(n)$ and $U(n)$ of $X$ such that
\begin{equation}\label{split}
X=S(n)\oplus U(n) \quad \text{for $n\in \Z$;}
\end{equation}
\item for each $n\in \Z$,
\begin{equation}\label{inv}
A_n S(n)\subset S(n+1) \quad \text{and} \quad A_n^{-1} U(n+1)\subset U(n);
\end{equation}
\item there exist $D, \lambda >0$ such that 
\[
| \mathcal A(m, n)x| \le De^{-\lambda (m-n)}|x| \quad \text{for $x\in S(n)$ and $m\ge n$,}
\]
and
\[
| \mathcal A(m, n)x| \le De^{-\lambda (n-m)}|x | \quad \text{for $x\in U(n)$ and $m\le n$;}
\]
\item we have that
\[
\sup_{n\in \Z} | P_n|<\infty, 
\]
where $P_n \colon X\to S(n)$ is a projection associated with the decomposition~\eqref{split}.
\end{itemize}
We observe that in contrast to the notion of an exponential dichotomy where 
\[
A_n S(n)=S(n+1) \quad \text{and} \quad A_n^{-1} U(n+1)=U(n),
\]
in the case of a generalized exponential dichotomy we impose a weaker condition~\eqref{inv}. We refer to~\cite[Section 5]{BD} for explicit examples of sequences that admit a generalized exponential dichotomy but do not admit an exponential dichotomy. We also note that when
$X$ is finite-dimensional, the notion of  a generalized exponential dichotomy coincides with the notion of an exponential trichotomy (which for the case of discrete dynamics, following~\cite{EH},  was first studied  in~\cite{P}).
Assuming that $\sup_{n\in \Z} \max \{ |A_n |, |A_n^{-1}| \} <+\infty$ and setting $\mathcal G(m,n)$ as in~\eqref{G-d}, one can easily obtain conditions under which Theorems~\ref{t1d} and~\ref{t2d} can be applied. We emphasize that those would be exactly the same as in the context of Remark~\ref{EDD}.

We note that when $X$ is finite-dimensional, it follows from~\cite[Proposition 1]{P} that $(A_n)_{n\in \Z}$ admits an exponential trichotomy if and only if the first estimate in~\eqref{rd} holds. 
\end{remark}

\section{Appendix}

\subsection{Solutions of~\eqref{Heq2} are defined for all time}\label{A1}
Let $t\mapsto y(t)$ be a solution of $\dot y=g(t,y)$. We know it is defined for
all $t$ and it is uniquely defined by its value at any time. Since $A(t)x+f(t,x,y(t))$ is locally
Lipschitz in $x$ it follows that the initial value problem
\[ \dot x=A(t)x+ f(t,x,y(t)),\quad x(\tau)=\xi \]
has a unique solution $x(t)$ for $t$ near $\tau$. Now we must show that this solution 
is defined for all time. Suppose $t\mapsto x(t)$ is not defined for all $t\ge \tau$. Then,  there exists $T>\tau$ such that $t\mapsto x(t)$ exists on $[\tau,T)$ 
but is not bounded. Let $t\mapsto z(t)$ be the solution of the inhomogeneous linear system
\[ \dot z=A(t)z+f(t,0,y(t)),\quad z(\tau)=x(\tau).\]
Then $t\mapsto z(t)$ exists and is bounded on $[\tau,T]$. Let the bound be $K$. We see by variation of constants that
\[ x(t)=z(t)+\int^t_{\tau}T(t,s)[f(s,z(s),y(s))-f(s,0,y(s))]\,ds.\]
Let $M>0$ be  such that
\[ |T(t,s)|\le M,\quad \tau\le s,t \le T.\]
By~\eqref{feq}, we have that 
\[ |x(t)| \le K+M\int^t_{\tau}\gamma(s)|x(s)|\, ds.\]
By Gronwall's lemma, it follows that for $\tau\le t<T$,
\[ |x(t)| \le Ke^{M\int^t_{\tau}\gamma(s)ds}\]
and hence  $t\mapsto x(t)$ is bounded on $[\tau, T)$, which yields a  contradiction. A similar argument works for $t\le \tau$.
So solutions of~\eqref{Heq2}  are defined for all $t$.

\subsection{Estimate for $\Delta_2(t,s)$}\label{A2}
Let $t\mapsto (x(t),y(t))$ and $t\mapsto (z(t),y(t))$ be solutions of (\ref{Heq2}). Then, by the variation of constants formula we have that 
\[\begin{split}
x(t)-z(t)
&=T(t,s)[x(s)-z(s)]\\ 
&\phantom{=} +\int^t_s T(t,u)[f(u,x(u),y(u))-f(u,z(u),y(u)]\, du.\end{split}\]
So if $t\ge s$, it follows from~\eqref{est-1} and~\eqref{EPs} that 
\[|x(t)-z(t)|\le K_1e^{a_1(t-s)}|x(s)-z(s)|+\int^t_sK_1e^{a_1(t-u)}\varepsilon|x(u)-y(u)|\,du.\]
Set $v(t)=e^{-a_1(t-s)}|x(t)-z(t)|$. Then,
\[ v(t)\le K_1|x(s)-z(s)|+K_1\varepsilon\int^t_sv(u)\, du.\]
By Gronwall's lemma, it follows that
\[ v(t)\le K_1|x(s)-z(s)|e^{K_1\varepsilon(t-s)}\]
and hence
\[ |x(t)-z(t)|\le K_1|x(s)-z(s)|e^{(a_1+K_1\varepsilon)(t-s)}.\]
The argument for $t\le s$ is similar.

\subsection{Estimate for $\Delta_3(t,s)$}\label{A3}
Let $t\mapsto (x(t),y(t))$ and
$t\mapsto (z(t),w(t))$ be solutions of (\ref{Heq2}) such that $x(s)=z(s)$. Then
\[ x(t)-z(t)=\int^t_s T(t,u)[f(u,x(u),y(u))-f(u,z(u),w(u))]\, du.\]
So if $t\ge s$, it follows from~\eqref{est-1}, \eqref{3:47} and~\eqref{M2} that 
\[ |x(t)-z(t)|\le \int^t_sK_1e^{M_3(t-u)}\varepsilon[|x(u)-z(u)|+|y(u)-w(u)|]\, du\]
and
\[ |y(t)-w(t)|\le e^{M_3(t-s)}|y(s)-w(s)|.\]
Then,
\[\begin{split}
 \int^t_se^{M_3(t-u)}|y(u)-w(u)|\, du
&\le \displaystyle \int^t_se^{M_3(t-u)}e^{M_3(u-s)}|y(s)-w(s)|\, du\\ 
&=(t-s)e^{M_3(t-s)}|y(s)-w(s)|.\end{split}\]
It follows that
\[ |x(t)-z(t)|\le K_1\varepsilon\int^t_se^{M_3(t-u)}|x(u)-z(u)|\, du
+K_1\varepsilon(t-s)e^{M_3(t-s)}|y(s)-w(s)|.\]
Set 
\[ v(t)=e^{-M_3(t-s)}|x(t)-z(t)|.\]
Then,
\[ v(t)\le K_1\varepsilon\int^t_sv(u)du+\lambda(t),\]
where
\[ \lambda(t)=K_1\varepsilon(t-s)|y(s)-w(s)|.\]
By Gronwall's lemma, we have 
\[ v(t)\le \lambda(t)+K_1\varepsilon\int^t_s\lambda(u)e^{K_1\varepsilon(t-u)}\, du.\]
Now, integrating by parts,
\[\begin{split}
& K_1\varepsilon\int^t_s\lambda(u)e^{K_1\varepsilon(t-u)}\, du\\ 
&= -e^{K_1\varepsilon(t-u)}\lambda(u)\bigg|^t_s
+\int^t_se^{K_1\varepsilon(t-u)}K_1\varepsilon|y(s)-w(s)|\, du\\ \\
&=-\lambda(t)+(e^{K_1\varepsilon(t-s)}-1)|y(s)-w(s)|.
\end{split}\]
Hence,
\[ v(t)\le (e^{K_1\varepsilon(t-s)}-1)|y(s)-w(s)|,\]
and so for $t\ge s$ we have that 
\[ |x(t)-z(t)|\le e^{(M_3+K_1\varepsilon)(t-s)}|y(s)-w(s)|.\]
Then
\[ |x(t)-z(t)|+|y(t)-w(t)|\le 2e^{(M_3+K_1\varepsilon)(t-s)}|y(s)-w(s)|.\]
The proof for $t\le s$ is similar. 

\medskip{\bf Acknowledgements.}
We would like to thank the referee for useful comments and suggestions. 
 L.B. was partially supported by a CNPq-Brazil PQ fellowship under Grant No. 306484/2018-8. D.D. was supported in part by Croatian Science Foundation under the project
IP-2019-04-1239 and by the University of Rijeka under the projects uniri-prirod-18-9 and uniri-pr-prirod-19-16.


\begin{thebibliography}{99}
\bibitem{AW} B. Aulbach and T. Wanner, \emph{Topological simplification of nonautonomous difference equations}, J. Difference Equ. Appl. \textbf{12} (2006), 283–296.
\bibitem{BD} L. Backes and D. Dragi\v cevi\' c, \emph{A generalized Grobman-Hartman theorem for nonautonomous dynamics}, preprint. 
\bibitem{BV} L. Barreira and C. Valls, \emph{H\"older Grobman-Hartman linearization}, Discrete Contin. Dyn. Syst. \textbf{18} (2007), 187--197.
\bibitem{Bel73}
G. R. Belitskii,\emph{ Functional equations and the conjugacy of diffeomorphism of finite smoothness class},
 Funct. Anal. Appl. {\bf 7} (1973), 268-277.
\bibitem{Bel78}
G. R. Belitskii,\emph{ Equivalence and normal forms of germs of smooth mappings}, Russian Math. Surveys {\bf 33} (1978),
107-177.
\bibitem{CR} A. Casta\~{n}eda and G. Robledo, \emph{
Differentiability of Palmer's linearization theorem and converse result for density function},
J. Differential Equations {\bf 259} (2015), 4634-4650.
\bibitem{C} W. A. Coppel, \emph{Dichotomies in Stability Theory}, Lect. Notes Math., vol. 629, Springer, Berlin/New York (1978).
\bibitem{CDS} L. V. Cuong, T. S. Doan and S. Siegmund, \emph{A Sternberg theorem for nonautonomous differential
equations},  J. Dynam. Diff. Eq. \textbf{31} (2019), 1279--1299.
\bibitem{DZZ} D. Dragi\v cevi\' c, W. Zhang and W. Zhang, \emph{Smooth linearization of nonautonomous difference equations with a nonuniform dichotomy}, Math. Z. \textbf{292} (2019), 1175--1193.
\bibitem{DZZ2} D. Dragi\v cevi\' c, W. Zhang and W. Zhang, \emph{Smooth linearization of nonautonomous differential equations with a nonuniform dichotomy}, Proc. Lond. Math. Soc. \textbf{121} (2020), 32--50.
\bibitem{E1}
M. S. ElBialy, \emph{Local contractions of Banach spaces and spectral gap conditions},  J. Funct. Anal. {\bf 182} (2001), 108-150.
\bibitem{E2} M. S. ElBialy,\emph{ Smooth conjugacy and linearization near resonant fixed points in Hilbert spaces},  Houston J. Math. \textbf{40} (2014), 467--509.
\bibitem{EH} S. Elaydi and O. Hajek, \emph{Exponential trichotomy of differential systems}, J. Math. Anal. Appl. \textbf{129} (1998), 362--374.
\bibitem{G1} D. Grobman, \emph{Homeomorphism of systems of differential equations}, Dokl. Akad. Nauk SSSR \textbf{128} (1959) 880--881.
\bibitem{G2} D. Grobman, \emph{Topological classification of neighborhoods of a singularity in $n$-space}, Mat. Sb. (N.S.) \textbf{56}  (1962), 77--94.
\bibitem{H1} P. Hartman, \emph{A lemma in the theory of structural stability of differential equations}, Proc. Amer. Math. Soc. \textbf{11} (1960) 610--620.
\bibitem{H2}  P. Hartman, \emph{On the local linearization of differential equations}, Proc. Amer. Math. Soc. \textbf{14} (1963) 568--573.
\bibitem{Jiang} L. Jiang, \emph{Generalized exponential dichotomy and global linearization}, J. Math. Anal. Appl. \textbf{315} (2006), 474--490.
\bibitem{Jiang2} L. Jiang,  \emph{Ordinary dichotomy and global linearization}, Nonlinear Anal. \textbf{70} (2009), 2722--2730.
\bibitem{Lin} F. Lin, \emph{Hartman's linearization on nonautonomous unbounded system}, Nonlinear Anal. \textbf{66} (2007), 38--50.
\bibitem{Palis} J. Palis, \emph{On the local structure of hyperbolic points in Banach spaces}, An. Acad. Brasil. Cienc. \textbf{40} (1968) 263–266.
\bibitem{Palmer}
K. Palmer, \emph{A generalization of Hartman's linearization theorem}, J. Math. Anal. Appl. {\bf 41} (1973), 753-758.
\bibitem{P} G. Papaschinopoulos, \emph{On exponential trichotomy of linear difference equations}, Applic.
Analysis \textbf{40} (1991), 89–109.
\bibitem{Pugh}C. Pugh, \emph{On a theorem of P. Hartman}, Amer. J. Math. \textbf{91} (1969) 363–367.
\bibitem{RS} A.A. Reinfelds and D. Steinberga, \emph{Dynamical equivalence of quasilinear equations}, Internat. J. Pure Appl. Math. \textbf{98} (2015), 355--364.
\bibitem{R-S-JDDE04} H. M. Rodrigues and J. Sol${\rm \grave{a}}$-Morales, \emph{Smooth linearization for a saddle on Banach spaces},
 J. Dyn. Differential Equations {\bf 16} (2004), 767-793.
\bibitem{R-S-JDDE06} H. M. Rodrigues and J. Sol${\rm \grave{a}}$-Morales, \emph{Invertible contractions and asymptotically stable
                    ODE's that are not $C^1$-linearizable},  J. Dyn. Differential Equations {\bf 18} (2006), 961-974.
\bibitem{SX} J.L. Shi and K. Q. Xiong, \emph{On Hartman's Linearization Theorem and Palmer's Linearization Theorem}, J. Math. Anal. Appl. \textbf{192} (1995), 813--832.

\bibitem{S} S. Siegmund, \emph{Dichotomy spectrum for nonautonomous differential equations}, J. Dynam. Differential Equations \textbf{14} (2002), 243--258.
\bibitem{Stern}
S. Sternberg, \emph{Local contractions and a theorem of Poincar\'e},   Amer. J. Math. {\bf 79} (1957), 809-824.
\bibitem{XWKR} Y. H. Xia, R. Wang, K. I. Kou and D. O'Regan, \emph{On the linearization theorem for nonautonomous differential equations}, Bull. Sci. Math. \textbf{139} (2015), 829--846.
\bibitem{ZhangZhang14JDE}
W. M. Zhang and W. N. Zhang, \emph{Sharpness for $C^1$ linearization of planar hyperbolic diffeomorphisms},
J. Differential Equations {\bf 257} (2014),
4470-4502.
\bibitem{ZZJ} W. M. Zhang, W. N. Zhang and W. Jarczyk, \emph{Sharp regularity of linearization for $C^{1, 1}$
hyperbolic diffeomorphisms},  Math. Ann. {\bf 358} (2014), 69-113.

\end{thebibliography}
\end{document}